\documentclass[12pt]{article}
\usepackage{amsmath,amsfonts,amsthm,amscd,amssymb,graphicx}
\usepackage{amssymb}
\usepackage{xcolor}
\def\rit{{\Bbb R}}
\def\cit{{\Bbb C}}
\def\nit{{\Bbb N}}

\def\eps{\varepsilon}
\def\G{{\mathcal G}}
\def\bel{\begin{equation*} \begin{aligned}}
\def\eel{\end{aligned} \end{equation*}}

\def\beln{\begin{equation} \begin{aligned}}
\def\eeln{\end{aligned} \end{equation}}

\def\cc{ \hbox{ c.c.}}

\newtheorem{theorem}{Theorem}[section]
\newtheorem{lemma}[theorem]{Lemma}
\newtheorem{e-proposition}[theorem]{Proposition}

\newtheorem{e-definition}[theorem]{Definition\rm}
\newtheorem{remark}{\it Remark\/}

\newtheorem{assumption}{\it Assumption\/}
\def\og{\leavevmode\raise.3ex\hbox{$\scriptscriptstyle\langle\!\langle$~}}
\def\fg{\leavevmode\raise.3ex\hbox{~$\!\scriptscriptstyle\,\rangle\!\rangle$}}

\def\beq{\begin{equation}}
\def\eeq{\end{equation}}

\def\X{\mathcal{X}}
\def\A{\mathcal{A}}

\begin{document}

\centerline{\Large \bf From linear to nonlinear instabilities }

\bigskip

\centerline{\Large \bf with application to plasma columns}

\bigskip

\bigskip

\centerline{Dongfen Bian\footnote{School of Mathematics and Statistics, Beijing Institute of Technology, $100081$ Beijing, China. Email: 
biandongfen@bit.edu.cn}}

\tableofcontents

%%%%%%%%%

\subsubsection*{Abstract}

%%%%%%%%%

This article addresses the problem of proving the nonlinear instability of an equilibrium
starting from its linear instability when there is no existence theory for the corresponding equations.
We design a general method based on the use of analytic functions to overcome this difficulty 
and apply it to the classical problem of the stability of a plasma column 
in magneto-hydrodynamics (MHD) as an illustration.

%%%%%%%%%%%%%%%%%%%%%%%

\section{Introduction}

%%%%%%%%%%%%%%%%%%%%%%%

In some cases we would like to prove that linear instability implies nonlinear instability; however there
exists no good existence theory for the corresponding equations, either because this question has not been investigated
yet, or because precisely the linear instabilities are too strong and prevent even the local existence of solutions.
In this paper we develop a specific method to address such issues and discuss a specific example,
the classical $\theta-$pinch and $z-$pinch instabilities of columns of plasma.

Let us first present the overall strategy. We consider an equation of the form
\beq \label{general1}
\partial_t u = {\cal L}(u)
\eeq
where $u(t,x)$ is a function and ${\cal L}$ is some nonlinear operator. Let $u_0$ be a stationary solution of (\ref{general1}),
namely a function that satisfies
\beq \label{general2}
{\cal L}(u_0) = 0.
\eeq
The linearized equation is 
\beq \label{general3}
\partial_t v = L v,
\eeq
where $L$ is the linear operator 
$$
L = d{\cal L}(u_0).
$$
In physics, the linear stability of stationary states is widely studied and often well-known. 
In many cases, the operator $L$ has also been studied in depth in mathematics, and in particular
its spectrum is well described.

In this paper, we assume that $u_0$ is linearly unstable (in a sense which will be made precise below) and 
prove that linear instability implies nonlinear instability under very general assumptions.

Let us introduce some notations. We write $u = u_0 + v$, then $v$ satisfies
\beq \label{general4}
\partial_t v = L v + Q(v)
\eeq
where $Q(v)$ contains quadratic and higher order terms. In order to simplify the presentation, we assume in this introduction
that $Q(v)$ is exactly quadratic, of the form $Q(v,v)$.
Let $\exp(\lambda_0 t) v_0$ be a linear instability, namely an exponentially growing solution of (\ref{general3}).
In particular
$$
L v_0 = \lambda_0 v_0, \qquad \Re \lambda_0 > 0.
$$
One possibility to construct a nonlinear instability starting from the linear one $\exp(\lambda_0 t) v_0$
is to follow the strategy detailed in \cite{Grenier}.
The first step is to construct a {\it formal} solution $v^{formal}$ of the form
\beq \label{formal}
v^{formal}(t,x) = \sum_{p \ge 1} \eps^p v_p(t,x)
\eeq
with
$$
v_1(t,x) = \exp(\lambda_0 t) v_0 + \cc,
$$
where $\cc$ stands for the complex conjugate and $\eps$ is a small parameter which 
measures the amplitude of the initial perturbation.

The functions $v_p$ for $p \ge 2$ are inductively constructed through
\beq \label{general5}
\partial_t v_p - L v_p = \sum_{1 \le q \le p - 1} Q(v_{q},v_{p - q}).
\eeq
The second step is to truncate (\ref{formal}) and to introduce an approximate solution $v^{app}$
defined by 
\beq \label{formal2}
v^{app}(t,x) = \sum_{p = 1}^N \eps^p v_p(t,x)
\eeq
where $N$ is large enough. Then $v^{app}$ satisfies (\ref{general4}) up to an error term 
of the form $\eps^{N+1} E^{app}(t,x)$ for some function $E^{app}$.
To finish the proof, we consider the solution $v(t)$ of (\ref{general4}) with initial data $v^{app}(0,x)$ and perform
an energy estimate on $v - v^{app}$. 

The strategy of \cite{Grenier} of course fails if we do not know whether $v$ exists or if we do not
know how to make an energy estimate on $v - v^{app}$.
In this case, one possibility is to prove that the whole infinite series (\ref{formal}) converges.

Usually $Q$ involves derivatives and each time we solve (\ref{general5}) we lose regularity.
As a consequence, we can only expect the convergence of (\ref{formal}) for analytic functions.
The techniques of generator functions \cite{Grenier2} appear to be very powerful to ensure such a convergence.
Note that the construction of the nonlinear instability is completely independent from the existence theory
of smooth solutions even if this theory exists. In particular, this strategy does not use energy estimates.

We detail this approach in the current paper. More precisely, in section $2$, we will introduce
specific spaces of analytic functions together with their related generator functions.
We will also detail the assumptions we need on the resolvent of $L$ (and more precisely on its Green function) and on $Q$ to prove that
linear instability implies nonlinear instability (see Theorem \ref{linearnonlinear}).
In section $3$, we apply this general theorem to a particular case, the stability of plasma columns
in the $\theta-$pinch and $z-$pinch configurations.

%%%%%%%%%%%%%%%%%%%%%%%%%%

\subsection*{Notations}

%%%%%%%%%%%%%%%%%%%%%%%%%%

Throughout this paper, the abbreviation $\cc$ stands for ``complex conjugate".
We will use cylindrical coordinates $(r,\theta,z)$.

%%%%%%%%%%%%%%%%%%%%%%%%%%%%%%%%%%%%%%%%%%%%%%%%%%%%%%%%%%%%

\section{A general instability theorem}

%%%%%%%%%%%%%%%%%%%%%%%%%%%%%%%%%%%%%%%%%%%%%%%%%%%%%%%%%%%%

Let us consider the general case of systems of the form
\beq \label{g1}
\partial_t u = {\cal L}(u) = A(u) + B(u) : \nabla_x u,
\eeq
where $x \in \rit^d$ and $u(t,x) \in \rit^m$.
In these equations, $A$ is an analytic function from $\rit^m$ to $\rit^m$
and 
$$
B(u) : \nabla_x u = \sum_{j=1}^d B_j(u) \, \partial_j u
$$
where $B_j$ are analytic functions from $\rit^m$ to $\rit^{m \times m}$.
Let $u_0(x)$ be an analytic stationary solution of (\ref{g1}), namely a solution of 
\beq \label{g3}
A(u_0) + B(u_0) : \nabla_x u_0 = 0.
\eeq
We are interested in the study of the nonlinear instability of (\ref{g1})
and more precisely we want to prove that the linear instability of (\ref{g1}) implies its nonlinear instability. By linear instability, we mean the existence of an exponentially
growing solution to the linearized system of (\ref{g1}).

In \cite{Grenier}, E. Grenier has developed a general methodology to construct nonlinear instabilities
starting from linear ones. His strategy relies on the construction of an approximate solution 
and on an energy estimate between the genuine solution and the approximate one.
Unfortunately, this strategy fails when there is no suitable energy estimate
or no existence theory for the system under study,
which is the case for the MHD system discussed in section $3$.

The construction of approximate solutions thus appears useless, and we have to directly construct
a genuine solution to the system. This genuine solution will be defined as an infinite series.
To ensure the convergence of this series, we need to work with analytic functions in order to overcome
the loss of derivatives coming from the nonlinearity, as will become clear later.
To this end, we will use generator functions \cite{Grenier2} in order to accurately control
the evolution of the radius of analyticity of our infinite series as time progresses.

Our method relies on two ingredients:

\begin{itemize}

\item a careful study of the resolvent of the linearized version of (\ref{g1}) in analytic spaces, 
based on the construction of its Green function, 
including cases where this Green function is singular near the boundary,

\item the construction of an exact solution through a series, whose convergence is controlled through
generator functions.

\end{itemize}

The use of Green functions to get bounds on the resolvent is new with respect to \cite{Grenier}.

%%%%%%%%%%%%%%%%%%%%%%%%%%%%%%%%%%%%%%%%%%%%%%%%

\subsection{Analytic spaces and generator functions}

%%%%%%%%%%%%%%%%%%%%%%%%%%%%%%%%%%%%%%%%%%%%%%%%

Let $\Omega \subset \mathbb{R}$.
For $u \in C^\infty(\Omega)$, we define the function $\G [u]$ on $\rho \ge 0$ by
$$
\G [u](\rho) = \sum_{\alpha \in \nit}  \| \partial_x^\alpha u \|_{L^\infty(\Omega)} {\rho^{\alpha} \over \alpha  !}.
$$
This definition can be extended to vector-valued functions by taking
the sum of the generator function of each component.

Note that the $L^\infty$ norm may be replaced by the norm of any Banach space.
For a sequence of functions $(v_n)_{n \ge 0}$, we define 
$$
\mathfrak{G} \Bigl[ (v_n)_{n \ge 0} \Bigr](\rho) = \sum_{n\ge 0} \| v_n \|_{L^\infty(\Omega)} {\rho^n \over n!}.
$$
We note that if $\G [u](\rho) < + \infty$, then $\G [u](\rho') < + \infty$ for any $0 \le \rho' \le \rho$.
Moreover $\G [u](0) = \| u \|_{L^\infty} < + \infty$. Hence there exists a largest real number $R(u) \ge 0$ such that
$\G [u](\rho) < + \infty$ if $0 \le \rho < R(u)$ and such that $\G [u](\rho)$ diverges if $\rho > R(u)$.
If $R(u) > 0$, we say that $u$ is analytic and the scalar $R(u)$ will be called the radius of analyticity of $u$.

We observe that if $u$ and $v$ are two functions of $L^\infty$ then
\beq \label{prop1}
\G[u+v](\rho) \le \G[u](\rho) + \G[v](\rho), 
\qquad 
\G [ uv] (\rho) \le \G[u](\rho) \, \G [v](\rho),
\eeq
and for any $\alpha$, we have
\beq \label{prop2}
\G [\partial_x^\alpha u] (\rho) = \partial_\rho^{\alpha} \G[u] (\rho).
\eeq
Moreover, $\mathcal{G}[u](\rho)$ is non decreasing and convex in $\rho$.
Let $F$ be an analytic function near $0$, then, locally, $F$ is the sum of its Taylor expansion
$$
F(v) = \sum_{n \ge 0} a_n v^n
$$
provided $|v| \le \rho_F$, where $\rho_F > 0$ is smaller than the radius of convergence of this series.
Then if $u \in L^\infty$,
\beq \label{prop3}
\G[ F(u)] (\rho) \le \sum_{n \ge 0} |a_n| \Bigl(\mathcal{G}[u](\rho) \Bigr)^n \le \| F \| \Bigl(\mathcal{G}[u](\rho) \Bigr),
\eeq
where $\| F \|$ is the function defined by
$$
\| F \| (\rho)  = \sum_{n \ge 0} |a_n| \, |\rho|^n.
$$
We note that $\| F \|$ converges provided $|\rho| \le \rho_F$.

%%%%%%%%%%%%%%%%%%%%%%%%%%%%%%

\subsection{Linear instability implies nonlinear instability}

%%%%%%%%%%%%%%%%%%%%%%%%%%%%%%

The linearized version of (\ref{g1}) around the equilibrium state $u_0$ with initial data $f_u$ reads
\beq \label{g5}
\partial_t u = Lu =  L_A \cdot u + L_B \cdot u
\eeq
with
$$
L_A \cdot u := \nabla_u A(u_0) \cdot u 
$$
and 
$$
L_B \cdot u := \Bigl[ \nabla_u B(u_0) \cdot u \Bigr] : \nabla_x u_0 + B(u_0) : \nabla_x u.
$$
We denote by $\tilde u$ the Laplace transform in time of the solution $u$ 
and by $\lambda$ the Laplace variable. This leads to the resolvent equation
\beq \label{g9}
( \lambda - L ) \tilde u = f_u,
\eeq
where $L = L_A + L_B$.
We denote by $\sigma$ the spectral abscissa of $L$, namely
\beq \label{definitionsigma}
\sigma = \sup_{\lambda \in Sp(L)} \Re \lambda.
\eeq
This leads to a first assumption.

\begin{assumption} The spectral abscissa $\sigma$ is finite and there exists an eigenvalue $\lambda_0$
with corresponding eigenvector $U_{\lambda_0}$ such that  
\beq \label{assumption1}
\Re \lambda_0 > {\sigma \over 2}
\eeq
and such that $\mathcal{G}[U_{\lambda_0}](\rho_0) < + \infty$ for some $\rho_0 > 0$.
\end{assumption}

We will also assume that, provided that $\Re \lambda$ is large enough, we can invert $\lambda - L$ 
in analytic spaces, which leads to the second assumption.

\begin{assumption}
If $\lambda \in \cit$ is such that $\Re \lambda \ge 2 \Re \lambda_0$, then $\lambda - L$ is
invertible and for any function $u$ and any $0 \le \rho \le \rho_0$ such that $\mathcal{G}[u](\rho) < + \infty$,
we have
\beq \label{assumption2}
\mathcal{G} \Bigl[ (\lambda - L)^{-1} u \Bigr](\rho) \le {C_0 \over | \lambda|} \mathcal{G}[u](\rho).
\eeq
\end{assumption}

The decay in $|\lambda|^{-1}$ of the inverse of the resolvent is classical and 
comes from the fact that, for large $\lambda$,
$(\lambda - L) \tilde u$ can be seen as a perturbation of $\lambda \tilde u$.
We now turn to the assumption on the nonlinearity $Q(u)$ defined by
$$
Q(u) = {\mathcal L}(u) - L u .
$$
We will assume that $Q$ is quadratic, with one derivative.

\begin{assumption} \label{assumptionQ}
We assume that the nonlinearity $Q(u)$ is quadratic, of the form $Q(u,v)$, with, for any $0 \le \rho \le \rho_0$
and for any functions $u$ and $v$,
$$
{\mathcal G}[Q(u,v)](\rho) \le C_Q \, {\mathcal G}[u](\rho) \, {\mathcal G}[\nabla_x v](\rho).
$$
\end{assumption}

Note that assumption $3$ is satisfied for instance if the nonlinearity is of the form 
$$
Q(u,v) = \widetilde Q(u,\nabla_x v)
$$
where $\widetilde Q$ is a quadratic form.
Our method can be extended to general nonlinearities $Q(u,v)$ which are sum of products of components of $\nabla_x v$ 
with holomorphic functions of $u$.

We now state our general instability result.

\begin{theorem} \label{linearnonlinear} (Linear instability implies nonlinear instability) \\
Let $u_0$ be an analytic stationary solution of (\ref{g1}). Then, if assumptions $1$, $2$ and $3$ hold true,
$u_0$ is nonlinearly unstable in the following sense: there exists $\eps_1 > 0$ and $\rho_1 > 0$ 
such that for any $\delta_1 > 0$
there exists a solution $u(t)$ to (\ref{g1}) on some time interval $[0,T]$ satisfying
\beq \label{insta1}
\mathcal{G}[ u(0) - u_0](\rho_1) \le \delta_1
\eeq
and
\beq \label{insta2}
\| u(T) - u_0 \|_{L^\infty} \ge \eps_1.
\eeq
% \beq \label{insta3}
% \| u(T) - u_0 \|_{L^\infty} \ge \eps_1, \qquad \| u(T) - u_0 \|_{L^1} \ge \eps_1.
% \eeq
\end{theorem}

\begin{remark}
Initially, $u - u_0$ is small in analytic spaces, and thus in any Sobolev spaces and in particular in $L^\infty$ (provided we are working in 
a bounded domain).
At time $T$, $u - u_0$ is large not only in analytic spaces, but also in all the Sobolev spaces and even
in $L^\infty$, which is the strongest possible instability: the initial perturbation is small in very
regular spaces and becomes large in very large spaces.
As will be clear in the proof, $T$ is of order $\log \delta_1^{-1}$.
\end{remark}

\begin{proof}
Let us look for a solution of the form $u = u_0 + v$.
We first formally construct a solution of the form
\beq \label{series}
v(t,x) = \sum_{n \ge 1} \sum_{p=-n}^n e^{n t \Re \lambda_0 + i p t \Im \lambda_0} u_{n,p}(x)
\eeq
and choose 
$$
u_{1,1}(x) = U_{\lambda_0}(x),
\quad u_{1,0} = 0,
\quad u_{1,-1} = \bar u_{1,1}.
$$
Note that the expansion (\ref{series}) is more complex than (\ref{formal}),
since in the introduction, for the sake of simplicity,
we restrict the presentation to the particular case where $\lambda$ is real. 

We will construct $u_{n,p}$ by induction.
As $Q$ is quadratic, we have
\beq \label{recu}
\Bigl[ ( n \Re \lambda_0 + i p \Im \lambda_0) - L \Bigr] u_{n,p}
= \sum_{1 \le n_1 \le n - 1} \sum_{p_1 = -n_1}^{n_1}
Q ( u_{n_1,p_1},u_{n-n_1,p-p_1}).
\eeq
We note that $u_{n,p} = 0$ if $|p| > n$.
Let
$$
G_n(\rho) = \sum_{p = -n}^n \G[u_{n,p}](\rho).
$$
Then, using assumptions $2$ and $3$,
\beq \label{ineq1}
G_n(\rho) \le {C \over n} \sum_{ 1 \le n' \le n-1} G_{n'}(\rho) \partial_\rho G_{n - n'}(\rho)
\eeq
for some positive constant $C$.
We now define
$$
F_N(\tau,\rho) = \sum_{1 \le n \le N} G_n(\rho) \tau^{n-1}.
$$
Multiplying (\ref{ineq1}) by $(n-1) \tau^{n-2}$, we obtain
$$
(n-1) \tau^{n-2} G_n(\rho) \le C \sum_{1 \le n' \le n-1} \tau^{n'-1} G_{n'}(\rho) \tau^{n - n'-1} \partial_\rho G_{n- n'}(\rho),
$$
thus
\beq \label{inequality}
\partial_\tau F_N \le C F_N \partial_\rho F_N.
\eeq
Hence $F_N$ satisfies an Hopf inequality that we will exploit in order to get bounds on $F_N$ which 
are uniform in $N$.

We note that $C_1 := 2 G_1(\rho_0) < + \infty$.
We will prove the following claim: provided $C_2$ is large enough, it holds that
\beq \label{claim}
F_N(\tau,\rho) \le C_1, \qquad \forall  0 \le \tau \le C_2^{-1} \rho_0, \qquad \forall  0 \le \rho \le \rho_0 - C_2 \tau.
\eeq
Let $\tau_N \le C_2^{-1} \rho_0$ be the largest time such that 
\beq \label{claim2}
F_N(\tau,\rho) \le C_1, \quad \forall 0 \le \tau \le \tau_N, \quad
\forall 0 \le \rho \le \rho_0 - C_2 \tau
\eeq
holds true. We have to prove that $\tau_N = C_2^{-1} \rho_0$.

First we note that $F_N(0,\rho) = G_1(\rho)$, thus $F_N(0,\rho) \le C_1 / 2$. 
By continuity, $F_N(\tau,\rho) \le C_1$ for any $0 \le \rho \le \rho_0$ provided $\tau$ is small enough,
thus $\tau_N > 0$.
We now introduce for $0 \le \rho \le \rho_0$,
$$
\widetilde F_N(\tau,\rho) = F_N \Bigl( \tau, \phi(\tau) \rho  \Bigr)
$$
where the function $\phi(\tau)$ will be chosen later.
We have
$$
\partial_\tau \widetilde F_N = \partial_\tau F_N + \phi'(\tau) \rho \partial_\rho F_N
$$
thus
\beq \label{Hopf}
\partial_\tau \widetilde F_N 
- \phi^{-1}(\tau)\Bigl[ \phi'(\tau) \rho + C \widetilde F_N \Bigr] \partial_\rho \widetilde F_N \le 0.
\eeq
Let us choose $\phi(\tau)$ such that $\phi'(\tau) = - 2 C \rho_0^{-1} C_1 < 0$ and $\phi(0) = 1$, namely
$$
\phi(\tau) = 1 - 2 C \rho_0^{-1} C_1 \tau.
$$
Let us consider $\widetilde F_N$ on $[0,\rho_0]$. Using (\ref{claim2}), 
the term between brackets is positive at $\rho = 0$ and negative at $\rho = \rho_0$. As a consequence,
the characteristics of  (\ref{Hopf}) are outgoing. Thus, using
the classical method of  characteristics, we see that $\widetilde F_N$ is bounded by
its initial value: $\widetilde F_N  \le C_1/2$.
We choose $C_2 = 2 C C_1$.
Then $F_N(\tau,\rho) \le C_1/2$ provided $0 \le \rho \le \rho_0 - C_2 \tau$.
If $\tau_N < C_2^{-1} \rho_0$, namely if $\phi(\tau_N) > 0$, by continuity (\ref{claim2}) is satisfied on a 
larger interval, which is contradictory. The claim is thus proved.

Now we let $N$ go to infinity, which gives
$$
F(\tau,\rho) \le C_1,   \qquad \forall 0 \le \tau \le {\rho_0 \over C_2}, \qquad \forall 0 \le \rho \le \rho_0 - C_2 \tau.
$$
Let us now prove that $v(t,x)$, given by (\ref{series}) is well defined. We have
\begin{align*}
\mathcal{G}[v](\rho) &\le \sum_{n \ge 1} e^{n t \Re \lambda_0} G_n(\rho)
\le e^{t \Re \lambda_0} F(e^{t \Re \lambda_0},\rho).
\end{align*}
Let us fix $\rho = \rho_0/2$. Then $v$ is well defined provided $e^{t \Re \lambda_0} \le (2 C_2)^{-1} \rho_0$,
namely provided
$$
t \le t^{\star} = {1 \over \Re \lambda_0} \log {\rho_0 \over 2 C_2}.
$$
Now, as $\kappa$ goes to $+\infty$, 
$$
v(t^\star - \kappa,x) \sim e^{(t^\star - \kappa) \lambda_0} u_{1,1}(x) + \cc,
$$
thus, there exists $\eps_1 > 0$ small enough
and $\kappa$ fixed, large enough, such that 
$$
\| v(t^{\star} - \kappa,\cdot) \|_{L^\infty} \ge \eps_1
$$
for some $\eps_1 > 0$ small enough.
Moreover, $\mathcal{G}[v(t)](\rho_0/2) < + \infty$ and goes to $0$ as $t$ goes to $-\infty$.
We now define 
$$
v_n(t,x) = v(t-n,x).
$$
We have $\mathcal{G}[v_n(0)](\rho_0/2) \to 0$ as $n \to + \infty$ and $\| v_n(t^{\star}-\kappa+n,\cdot) \|_{L^\infty} \ge \eps_1$,
which proves the Theorem.
\end{proof}

\begin{remark}
The proof can be extended to the case where $Q$ is not quadratic but contains higher order terms (cubic, quartic,...)
provided it involves only one derivative.
We then expand it as a power series.
The recurrence relation (\ref{recu}) is then more complex and involves cubic, quartic, and higher polynomials
in the various $u_{n,p}$. 
\end{remark}

\subsection{Resolvent and Green functions}

%%%%%%%%%%%%%%%%%%%%%%%%

We now detail how the use of Green functions allows us to prove the assumption $2$, 
a point that has not been investigated in \cite{Grenier2}. 
The following Lemma reduces the study of the resolvent equation to the study of the generator function of its Green function.

We first introduce a notation. We denote by $[\partial^n G(x)]$ the jump of $\partial^n G(x,y)$ at $x = y$,
namely
$$
[ \partial^n G(x)] = \partial_x^n G(x^+,x) - \partial_x^n G(x^-,x).
$$

\begin{lemma} \label{green}
Let $\Omega = \rit$.
Let $G(y,x)$ be the Green function of $\lambda - L$. Assume that $G$ satisfies
\beq \label{AAA0}
\sum_{n \ge 0} \Bigl[ \sup_{x \in \rit} \int_{\rit} | \partial_x^n G(y,x) | \, dy \Bigr] {\rho_G^n \over n!} \le C_0,
\eeq
\beq \label{AAA}
 \sum_{n \ge 0}  \sup_{x \in \rit}
 \Bigl| [\partial_x^n G(x) ] \Bigr| {\rho_G^n \over n!} \le C_0
\eeq
for some $\rho_G > 0$ and some $C_0 > 0$.
Then the solution $u(x)$ of
$$
(\lambda - L) u = f
$$
is defined by
\beq \label{defiu01}
u(x) = \int_\rit G(y,x) f(y) \, dy
\eeq
and satisfies
\beq \label{conclusion}
\| u \|_{L^\infty} + \mathcal{G}[\partial_x u](\rho) \le C_0 \mathcal{G}[f](\rho)
\eeq
for any $\rho \le \rho_G/2$ and any function $f$. 
\end{lemma}

\begin{proof}
First, we have
$$
\| u \|_{L^\infty} \le \| f \|_{L^\infty} \int_\rit | G(y,x) | \, dy
$$
which allows to bound $\| u \|_{L^\infty}$ using (\ref{AAA0}).
Differentiating the definition of $u$, we have
$$
\partial_x u(x) = \int_{-\infty}^x \partial_x G(y,x) f(y) \, dy + \int_x^{+\infty} \partial_x G(y,x) f(y) \, dy
+ [G(x)] f(x)
$$
where 
$$
[G(x)] = G(x^-,x) - G(x^+,x)
$$ 
is the difference between the limits of the left and right hand sides of $G(y,x)$ at $y = x$.
Iterating, we have
\begin{align*}
\partial_x^2 u(x) =& \int_{-\infty}^x \partial_{xx} G(y,x) f(y) \, dy + \int_x^{+\infty} \partial_{xx} G(y,x) f(y) \, dy
\\&
+ \partial_x \Bigl( [G(x)] f(x) \Bigr) + [\partial_x G(x)] f(x)
\end{align*}
where 
$$
[\partial_x G(x)] = \partial_x G(x^-,x) - \partial_x G(x^+,x).
$$
We define the sequences 
$U = (U_n)_{n \ge 0} := (\partial_x^{n+1} u(x))_{n \ge 0}$, $H = (H_n)_{n \ge 0}$ where
$$
H_n = \int_{-\infty}^x \partial_x^{n+1} G(y,x) f(y) \, dy + \int_x^{+\infty} \partial_x^{n+1} G(y,x) f(y) \, dy
$$
and, for $p \ge 0$, $F_p = (F_{p,n})$ where $F_{p,n} = 0$ if $p > n$ and
$$
F_{p,n}(x) = \partial_x^{n -p} \Bigl( [\partial_x^{p} G] f \Bigr)
$$
if $p \le n$.
We then have
$$
U = H + \sum_{p \ge 0} F_p,
$$
hence
$$
\mathfrak{G} \Bigl[ (U_n)_{n \ge 0} \Bigr](\rho) \le \mathfrak{G} \Bigl[ (H_n)_{n \ge 0} \Bigr](\rho) 
+ \sum_{p \ge 0} \mathfrak{G} \Bigl[ (F_{p,n})_{n \ge 0} \Bigr](\rho) .
$$
The first term can be estimated as 
$$
\mathfrak{G} \Bigl[ (H_n)_{n \ge 0} \Bigr](\rho) 
\le C \| f \|_{L^\infty} 
\mathfrak{G} \Bigl[ \Bigl\{ \sup_x \int_\rit |\partial_x^{n+1} G(y,x)| \, dy \Big\}_{n \ge 0} \Bigr](\rho)
$$ 
which can be bounded using (\ref{AAA0}).
Note that,  
\begin{align*}
\mathfrak{G} \Bigl[ \Bigl\{ \sup_x \int_\rit |\partial_x^{n+1} G(y,x)| \, dy \Big\}_{n \ge 0} \Bigr](\rho)
&= \sum_{n \ge 0} \sup_x \int_\rit |\partial_x^{n+1} G(y,x)| \, dy{\rho^n \over n!}
\\
&\le C\sum_{n \ge 0} \sup_x \int_\rit |\partial_x^{n+1} G(y,x)| \, dy
{\rho_G^{n+1} \over (n+1)!}
\end{align*}
for some positive constant $C$
provided $\rho \le \rho_G / 2$ since
$$
{\rho^n \over \rho_G^{n+1}} {(n+1)! \over n!} 
$$
is then uniformly in $n$.

For the second term, we have 
\begin{align*}
\mathfrak{G} \Bigl[ (F_{p,n})_{n \ge 0} \Bigr](\rho) 
&\le \sum_{n \ge p} \Bigl\| \partial_x^{n-p} ( [\partial_x^{p} G] f ) \Bigr\|_{L^\infty} {\rho^n \over n!}
\cr
&\le {\rho^p \over p!} \sum_{n \ge p} \Bigl\| \partial_x^{n-p} ( [\partial_x^{p} G] f ) \Bigr\|_{L^\infty} {\rho^{n-p} \over (n - p)!}
\\
&\le {\rho^p \over p!} \mathcal{G} \Bigl[ [\partial_x^{p} G] f \Bigr](\rho)
\le {\rho^p \over p!} \mathcal{G} \Bigl[ [\partial_x^{p} G] \Bigr] \, \mathcal{G}[f](\rho).
\end{align*} 
Moreover,
\begin{align*}
\sum_{p \ge 0} {\rho^p \over p!} \mathcal{G} \Bigl[ [\partial_x^{p} G] \Bigr]
&\le  \sum_{p \ge 0} \sum_{n \ge 0} {\rho^p \over p!} { \| [\partial_x^{p+n} G] \|_{L^\infty} \over n!} \rho^n
\\
&\le  \sum_{p \ge 0} \sum_{n \ge 0} { \| [\partial_x^{p+n} G] \|_{L^\infty} \over (n +p )!} \rho^{n+ p }
{ (n + p )! \over p! \, n!} 
\\
&\le   \sum_{q \ge 0} { \| [\partial_x^q G] \|_{L^\infty} \over q! } 
\rho^q \sum_{0 \le p \le q} {q! \over p! (q - p)!}
\\
&\le   \sum_{q \ge 0} { \| [\partial_x^q G] \|_{L^\infty} \over q!} 2^{q} \rho^q
\\
&\le  \,  {\mathfrak G} \Bigl[ [ \partial_x^n G]_{n \ge 0}  \Bigr] (2 \rho) 
\end{align*}
which is finite by assumption. This ends the proof with the set of assumptions (\ref{AAA0}) and (\ref{AAA}).
\end{proof}

\begin{remark}
If $L$ is a second order elliptic operator, then $[ G(x)] = 0$. In this case, the term $[ G(x)] f(x)$ vanishes and, in addition to (\ref{conclusion}), we have
$$
{\mathcal G}[\partial_x^2 u] (\rho) \le C {\mathcal G}[f](\rho)
$$
which translates the usual gain of two derivatives when we invert a one dimensional second order elliptic operator.
\end{remark}

We now turn to the case where $\lambda$ is large. 
The aim of this paragraph is to obtain the factor $\lambda^{-1}$ of Assumption $2$.
We will only consider a particular case, on the whole line $\rit$, when $L - \lambda$ is of the form
\beq \label{example2}
\partial_x^2 u  + q(x) u = \lambda u
\eeq
where $q(x)$ is a smooth analytic function. We assume that $q$ converges to some constants $q_\pm$ at $\pm \infty$.
The Green function $G$ is then the solution of
\beq \label{Glambda}
\partial_y^2 G + (q - \lambda) G = \delta_x
\eeq
which goes to $0$ at $\pm \infty$. We now explicitly construct this Green function $G$.

We first construct two solutions $\psi_\pm$ of 
$$
\partial_y^2 \psi_\pm + (q - \lambda) \psi_\pm = 0
$$
using the classical WKB solution.
Namely, we search for $\psi_\pm$ in the form
$$
\psi_\pm(x) = \exp\!\left( \int^x_{x_0} S_\pm(t)\,dt \right)
$$
where $x_0$ can be arbitrarily chosen,
which leads to the Riccati equation for $S_\pm$
$$
S_\pm' + S_\pm^2 = \lambda - q(x).
$$
It is classical to prove that 
$$
S_\pm(x) = \pm \sqrt{\lambda} \mp \frac{q(x)}{2 \sqrt{\lambda}} + O(\lambda^{-1}),
$$
which leads to 
$$
\psi_\pm(x) = \exp\!\left( \pm \sqrt{\lambda}\,x \mp \frac{1}{2 \sqrt{\lambda}}\int_{x_0}^x q(t)\,dt + O(\lambda^{-1}) \right)
= e^{\pm \sqrt{\lambda} x} \theta_\pm(\lambda,x).
$$
We note that the Wronskian $W$ between $\psi_+$ and $\psi_-$ equals $2 \sqrt{\lambda} + O(1)$.

The Green function $G$ is explicitly given by
$$
G (y,x) = {1 \over W} \left\{ \begin{array}{l} \psi_+(x) \psi_-(y) \quad \hbox{if } x < y, \cr
\psi_-(x) \psi_+(y) \quad \hbox{if } x > y.
\end{array} \right.
$$
We note that
$$
G (y,x) = {e^{- \sqrt{\lambda}| x - y|}  \over W} \theta_+(\lambda,x) \theta_-(\lambda,y)
$$
if $x < y$ and similarly if $x > y$.
Moreover, $\| G \|_{L^\infty}$ is small, of order $\lambda^{-1/2}$ since $W$ is of order $\sqrt{\lambda}$
in $L^\infty$. 
We now write, with $z = x - y$,
$$
\int_{-\infty}^{+\infty} G(y,x) f(y) \, dy 
= \int_{-\infty}^{+\infty} {e^{- \sqrt{\lambda} |z|} \over W} 
\theta_+(\lambda,x) \theta_-(\lambda, x - z) f(x-z) \, dz
$$
and repeat the proof of Lemma \ref{green} to get
$$
\| u \|_{L^\infty} + \mathcal{G}[\partial_x u](\rho) \le {C_0 \over |\lambda|} \mathcal{G}[f](\rho).
$$

%%%%%%%%%%%%%%%%%%%%%%%%%%%%%%%%%%%%%%%%

\subsection{Singularities and boundaries \label{secsingular}}

%%%%%%%%%%%%%%%%%%%%%%%%%%%%%%%%%%%%%%%%

The functions under study may be singular at the boundary of their definition domain. 
In this case, weights need to be introduced in the definition of the generator function.
We will focus on functions defined on $0 < x < 1$ that are singular at $x = 0$. The method can be extended
to functions that are singular both at $x = 0$ and $x = 1$ and to other possible domains $\Omega$.

Let us consider functions such that 
$$
|\partial_x^\alpha u| \lesssim x^{- \alpha}.
$$
For such functions, we introduce the modified generator function
$$
{\mathcal G}_{sing}[u](\rho) = \sum_{\alpha \ge 0} \| (x \partial_x)^\alpha u \|_{L^\infty} {\rho^{|\alpha|} \over |\alpha|!},
$$
for which we have (\ref{prop1}), (\ref{prop3}) and
\beq \label{prop2bis}
\mathcal{G}_{sing}[ x \partial_x u](\rho) = \partial_\rho \mathcal{G}_{sing}[u](\rho).
\eeq
We also have an analogue of Lemma \ref{green}, as stated below.

\begin{lemma} \label{greensingular}
Let $G(y,x)$ be a function defined for $x,y\in(0,1)$, analytic in $x$ for $x \neq y$ with respect to the operator $x\partial_x$, i.e.,
\beq \label{singular1}
\sum_{n \ge 0} \sup_{(0,1)} \int_0^1 \bigl| (x\partial_x)^n G(y,x) \bigr| \, dy \, \frac{\rho_G^n}{n!} \le C_0,
\eeq
\beq \label{AAA2}
 \sum_{n \ge 0}  \sup_{x \in (0,1)}
 \Bigl| [(x \partial_x)^n G(x) ] \Bigr| {\rho_G^n \over n!} \le C_0
\eeq
for some $\rho_G > 0$ and some constant $C_0$. Then the function $u$, defined by
$$
u(x) = \int_0^x G(y,x) f(y) \, dy + \int_x^1 G(y,x) f(y) \, dy,
$$
satisfies
$$
\| u \|_{L^\infty} + \mathcal{G}_{sing}[\partial_x u](\rho) \le C \mathcal{G}_{sing}[f](\rho)
$$
for any $0 \le \rho \le \rho_G/2$ and some constant $C$ independent of $\rho$ and $f$.
\end{lemma}

\begin{proof}
We make the change of variables $s = \ln x$, $t = \ln y$, so that $x = e^s$, $y = e^t$ with $s,t \in (-\infty,0]$. Then
$$
x\partial_x = \partial_s, \qquad (x\partial_x)^n G(y,x) = \partial_s^n \widetilde G(t,s),
$$
where $\widetilde G(t,s) = G(e^t, e^s)$. The generator becomes
$$
\mathcal{G}_{sing}[u](\rho) = \sum_{n \ge 0} \| \partial_s^n \widetilde u(s) \|_{L^\infty} \frac{\rho^n}{n!},
$$
where $\widetilde u(s) = u(e^s)$. The integral expression for $u$ transforms into
$$
\widetilde u(s) = \int_{-\infty}^{s} \widetilde G(t,s) \, \widetilde f(t) \, e^t dt + \int_{s}^{0} \widetilde G(t,s) \, \widetilde f(t) \, e^t dt,
$$
with $\widetilde f(t) = f(e^t)$. This formula is similar to (\ref{defiu01}), up to the exponential factor $e^t$
which is bounded by $1$ and integrable on $(-\infty,0)$.
We thus repeat the proof of Lemma~\ref{green}.
Thanks to the weight $e^t$, we have
$$
| H_n | \le \sup_{t \ne s} | \partial_s^{n+1} \widetilde G| \, \| \widetilde f\|_{L^\infty}
$$
thus the bounds on $H$ are similar.  Next, 
$$
F_{p,n} = \partial_s^{n-p} \Bigl( [\partial_s^p \widetilde G] \widetilde f e^s \Bigr), 
$$
and we can repeat the arguments of Lemma \ref{green} on $\widetilde f(t) e^t$ instead of $f(x)$. 
The computations are similar, except that we obtain
$$
\mathfrak{G} \Bigl[ (F_{p,n})_{n \ge 0} \Bigr] (\rho) \le {\rho^p \over p!}
\mathcal{G}\Bigl[ [\partial_t^p G] \Bigr] \mathcal{G}[\widetilde f e^t](\rho).
$$
However
$$
\mathcal{G}[\widetilde f e^t](\rho) \le \mathcal{G}[\widetilde f](\rho) \mathcal{G}[e^t](\rho)
$$
and $\mathcal{G}[e^t] < + \infty$.
Thus, repeating the proof of Lemma \ref{green}, we obtain 
$$
\sum_{n \ge 0} \| \partial_s^{n+1} \widetilde u \|_{L^\infty} \frac{\rho^n}{n!} 
\le C \sum_{n \ge 0} \| \partial_s^n \widetilde f \|_{L^\infty} \frac{\rho^n}{n!}
$$
for $0 \le \rho \le \rho_G/2$. The inverse change of variable, from $s$ to $x$ gives the desired estimate.
\end{proof}

\begin{remark}
When dealing with such singularities, the generator function $\mathcal{G}$ must be replaced
by $\mathcal{G}_{sing}$ in all the assumptions $1$ to $3$ of the Theorem \ref{linearnonlinear}.
\end{remark}

%%%%%%%%%%%%%%%%%%%%%%%%%%%%%%%%%%%%%%%%%%%%%%%%%%%%%%%%%%%%%%%%%%%%%%%%%%%%%%%%%%%%%%%%%%

\section{Application to the stability of plasma columns}

%%%%%%%%%%%%%%%%%%%%%%%%%%%%%%%%%%%%%%%%%%%%%%%%%%%%%%%%%%%%%%%%%%%%%%%%%%%%%%%%%%%%%%%%%%

%%%%%%%%%%%%%%%%%%

\subsection{Introduction}

%%%%%%%%%%%%%%%%%%

Let us consider the evolution of a plasma which is confined in a cylinder. 
We closely follow the presentation of J. P. Goedbloed and S. Poedts \cite{Goed}, and
in particular their chapter $9$.
We will assume that the plasma may be described by its density $\rho$, velocity $u$ and pressure $p$.
We will neglect the electric field and just focus on the magnetic field $B$, assuming moreover that
the speed of light is infinite.
We will also neglect the viscosity of the plasma and the diffusivity of the magnetic field.
Let $r_{ext} > 0$ and let
$$
\Omega_{cyl} = \Bigl\{ (x_1,x_2,z) \in \mathbb{R}^2
\times  \mathbb{T} \, \, | \, \,  
x_1^2 + x_2^2 \le r_{ext}^2 \Bigr\}
$$
be the containing cylinder. For simplicity, we assume periodic boundary conditions in $z$, but the infinite cylinder
$z \in \mathbb{R}$ can be treated similarly.

Let $\Omega(t) \subset \Omega_{cyl}$ be the subdomain occupied by the plasma
and let $\Omega^v(t) = \Omega_{cyl} - \Omega(t)$ be the ``vacuum" part
$$
\Omega(t) = \Bigl\{ \rho(t,x) > 0 \Bigr\},
\qquad
\Omega^v(t) = \Bigl\{ \rho(t,x) = 0 \Bigr\},
$$
where $\rho(t,x)$ is the density of the plasma.
We assume that $\Omega(t)$ is a small perturbation of $x_1^2 + x_2^2 = r_0^2$ with $r_0 < r_{ext}$.
We will also consider the case where $\Omega(t) = \Omega_{cyl}$. In this latter case, $r_0 = r_{ext}$ and there is no vacuum area.

In the domain $\Omega(t)$, the density $\rho$, the velocity $u$, the pressure $p$
and the magnetic field $B$ satisfy the classical ideal MHD equations \cite{Goed}
\beq \label{M1}
\partial_t \rho + \nabla \cdot (\rho u) = 0,
\eeq
\beq \label{M2}
\rho \Bigl( \partial_t u + (u \cdot \nabla ) u \Bigr) 
+ \nabla \Bigl( p + {1 \over 2} |B|^2 \Bigr) = (B \cdot \nabla) B,
\eeq
\beq \label{M3}
\partial_t B - \nabla \times (u \times B) = 0,
\eeq
\beq \label{M4}
\nabla \cdot B = 0.
\eeq
We assume that the pressure is a function
of the density only (so-called barotropic fluid) and focus on the isentropic case, namely, we assume that
\beq \label{pressure}
p(\rho) = A \rho^\gamma
\eeq
where $\gamma > 1$ is a given real number. 

%Moreover, (\ref{M4}) is propagated by (\ref{M3}) provided it is satisfied by the initial condition $B(0,x)$.
%The equations (\ref{M1})-(\ref{M3}) are then of ``hyperbolic" type.

In $\Omega^v(t)$, $\rho = 0$ by definition, $u$ is not defined
and the magnetic field $\widehat B$ satisfies
\beq \label{M5}
\nabla \cdot \widehat B = 0,
\eeq
\beq \label{M6}
\nabla \times \widehat B = 0.
\eeq
Let us now detail the boundary conditions (see \cite{Goed}, chapter $6.6$).
Assuming that $\partial \Omega_{cyl}$ is an ideal conductor, we have
\beq \label{M7}
n \cdot \widehat B = 0.
\eeq
% Moreover, further assuming that there is no current at the surface of the conductor, 
% \beq \label{M9bis}
% n \times \widehat B = 0.
% \eeq
% Thus, $\widehat B$ vanishes on the outer boundary $\partial \Omega_r$.
We now detail the conditions at the plasma-vacuum boundary
$$
\Sigma(t) = \overline{\Omega(t)} \cap \overline{\Omega^v(t)}.
$$
Let $n(x)$ be the normal to $\Sigma(t)$ at the point $x \in \Sigma(t)$.
%We refer to \cite{Goed} (section $4.5.2$) and \cite{Landau} for a discussion of these boundary conditions.
We have
\beq \label{M8}
B \cdot n = \widehat B \cdot n = 0,
\eeq
\beq \label{M9}
\Bigl[ \Bigl[ p + {|B|^2 \over 2} \Bigr] \Bigr] = 0
\eeq
where $[[\cdot]]$ is the jump at the boundary $\Sigma(t)$ with the convention
$$
[[B]] = B_{vacuum} - B_{plasma}.
$$
Note that the tangential component $B_t$ of the magnetic field may be discontinuous 
on $\Sigma(t)$.
This discontinuity is compensated by a surface current $j_\Sigma$ on $\Sigma(t)$ which satisfies
\beq \label{M9bis}
j_\Sigma = n \times [[ B ]].
\eeq
It remains to describe the dynamics of $\Sigma(t)$.
The normal velocity $\nu(\Sigma(t))$ of the interface $\Sigma(t)$ satisfies
\beq \label{M10}
\nu(\Sigma(t)) = u \cdot n.
\eeq
In this paper, we will refer to the set of equations 
(\ref{M1},\ref{M2},\ref{M3},\ref{M4},\ref{M7},\ref{M8},\ref{M9},\ref{M10}) 
as ``the MHD system".

\medskip

In this paper we prove that linear instability implies nonlinear instabilities for the MHD system
(see Theorems~\ref{nonl} and \ref{non2} for detailed statements). We refer to \cite{Goed} for a physical discussion
of the linear instabilities of plasma columns and to \cite{BGT1,BGT2} for a mathematical study of
these instabilities. For a comprehensive derivation of the ideal MHD equations and a detailed
discussion of plasma–vacuum interfaces, see the monographs \cite{Goed,Goedbloed2010}.
The classical linear stability analysis of cylindrical plasma columns, which leads to the
Hain–Lüst equation, is presented in \cite{HainLust1958} and further developed for free boundary
problems in \cite{Goedbloed1984}. The mathematical well‑posedness of the plasma–vacuum interface
problem has been the subject of intense investigation: the compressible case was for instance treated
in \cite{SecchiTrakhin2014,Morando2024} and the incompressible setting in
for instance \cite{HaoLuoMasmoudi2023}. We also refer to
\cite{Gu,Jiang1,Jiang2,HaoLuo,SecchiTrakhin2016,SecchiTrakhinWang2026,Xin,Xin2}. The present work
focuses on the nonlinear instability of stationary solutions.

%%%%%%%%%%%%%%%%%%%%%%%%%%%%%%%%%%%%%%%%%%%%%%%%%%%%%%%%%%%%%%%%%%%%%%%%

\subsection{Stationary solutions \label{stationary}}

%%%%%%%%%%%%%%%%%%%%%%%%%%%%%%%%%%%%%%%%%%%%%%%%%%%%%%%%%%%%%%%%%%%%%%%%

Following \cite{BGT1}, \cite{BGT2} and \cite{Goed} (in particular its Chapter $9$), 
for solutions with cylindrical symmetry, the MHD equations reduce to
$$
j \times B = \nabla p,
\qquad j = \nabla \times B,
\qquad
\nabla \cdot B = 0,
$$
namely, in cylindrical coordinates
$$
\partial_r p = j_\theta B_z - j_z B_\theta,
\qquad j_\theta = - \partial_r B_z,
\qquad j_z = {1 \over r} \partial_r (r B_\theta).
$$
Combining these identities, we obtain
\beq \label{cyl}
\partial_r \Bigl( p + {1 \over 2} ( B_\theta^2 + B_z^2) \Bigr)
+ {B_\theta^2 \over r} = 0.
\eeq
In particular, we can choose arbitrarily two of the three functions $B_\theta$,
$B_z$ and $p$ and define the last one through (\ref{cyl}), under the
natural constraint that $p$ must be non negative.
This leads to a large variety of possible stationary solutions of the
MHD equations.

At $r_0$, we must have the pressure balance
$$
p(r_0) + {1 \over 2} \Big[ B_\theta^2(r_0) + B_z^2(r_0) \Bigr]
= {1 \over 2} \Bigl[ \widehat B_{\theta}^2(r_0) + \widehat B_z^2(r_0) \Bigr].
$$
Note that the magnetic field may be discontinuous at $r = r_0$ because
of surface currents $j_\theta^s$ and $j_z^s$ given by
$$
j_\theta^s = B_z(r_0) - \widehat B_z(r_0),
\qquad
j_z^s = - B_\theta(r_0) + \widehat B_\theta(r_0).
$$
However, in this article we restrict ourselves to cases without surface currents.

\medskip

We now discuss two particular cases: the $\theta-$pinch configuration where $B_\theta =0$
and the $z-$pinch configuration where $B_z = 0$.

\begin{itemize}

\item $\theta-$pinch: in this case, $B_\theta = 0$ and thus
$$
p(r) + {1 \over 2} B_z^2(r) = C_0
$$
is a constant independent of $r$.
We can thus arbitrarily choose a constant $C_0$ and a smooth function $B_z$ 
such that $B_z^2(r_{0}) < 2 C_0$ and define $p$ accordingly.
Note that
$$
\rho(r) = A^{-1/\gamma} \Bigl( C_0 - {1 \over 2} B_z^2(r) \Bigr)^{1/\gamma}.
$$
More precisely, we will consider stationary solutions satisfying the following assumption

\begin{assumption} \label{assumptionAtheta}
We will assume that $B_z$ is analytic, namely
$$
{\mathcal G}[B_z](\rho_1) < + \infty
$$
for some positive $\rho_1$. We will also assume that $p$ does not
vanish on $0 \le r \le r_0$. We also have ${\mathcal G}[p](\rho_1) < + \infty$.
\end{assumption}

\item $z-$pinch: in this case, $B_z = 0$. 
The construction has been detailed in \cite{BGT1,BGT2}.
Let us describe the particular case where the current $j_z(r)$ inside the plasma is constant. 
If $j_z(r) = 2 c / r_{0}$, then
\beq \label{particular}
p(r) = c^2 \Bigl(1 - {r^2 \over r_{0}^2} \Bigr), \quad B_\theta(r) = c {r \over r_{0}},
\quad
\rho(r) =  \Bigl({c^2 \over A}\Bigr)^{1/\gamma} \Bigl( 1 - {r^2 \over r_{0}^2} \Bigr)^{1 / \gamma},
\eeq
and there is no surface current.

For this particular case, we have ${\mathcal G}[p](\rho_1) < + \infty$ for some $\rho_1 > 0$, and similarly
for $B_\theta$. However, $\rho$ is singular at $r_{0}$.
We will avoid this singularity by assuming that $p$ does not vanish and
study the stability of solutions which satisfy the following assumption.

\begin{assumption} \label{assumptionAZ}
We assume that $B_\theta$ is analytic, namely ${\mathcal G}[B_\theta](\rho_1) < + \infty$,
that $B_\theta$ vanishes at $r = 0$ and satisfies $\mathcal{G}[r^{-1} B_\theta](\rho_1) < + \infty$.
We assume that the pressure $p$ does not vanish on $0 \le r \le r_{0}$ and
satisfies ${\mathcal G}[p](\rho_1) < + \infty$. 
\end{assumption}

\end{itemize}

In this paper, we focus on the stability of these two classes of solutions.
We will discuss the particular example (\ref{particular}) in section \ref{secparticular}
to prove that the inviscid MHD system is not locally well-posed near this particular stationary solution.

Under assumptions \ref{assumptionAtheta} and \ref{assumptionAZ}, the local in time existence of smooth solutions
is a consequence of the results of \cite{SecchiTrakhin2014} and \cite{SecchiTrakhin2016}, but we do not need
these results to obtain nonlinear instability. The construction of the nonlinear instability is independent
from the existence theory of smooth solutions.

%%%%%%%%%%%%%%%%%%%%%%%%%%%%%

\subsection{The Lagrangian coordinates}

%%%%%%%%%%%%%%%%%%%%%%%%%%%%%

Let us first define the Lagrangian coordinates inside the plasma.
Let $h(t,\X)$ be the position at time $t$ of a plasma particle $\X$,
defined by
$$
{d \over dt} h(t,\X) = u \Bigl( t,h(t,\X) \Bigr), 
\qquad t > 0, \quad \X \in \overline{\Omega},
$$
with initial data 
$$
h(0,\X) = \X + \xi_0(\X)
$$
where $\xi_0$ is the initial displacement.
We define the displacement $\xi(t,\X)$ by
$$
\xi(t,\X) = h(t,\X) - \X.
$$
In Lagrangian coordinates, the plasma fills $\Omega_0 = \{ r \le r_0 \}$ and the interface
$\Sigma_0 = \{ r = r_0 \}$ is fixed.
In order to simplify the notation, we still denote by $\rho$, $u$, $p$ and $B$ the
density, velocity, pressure and magnetic field in Lagrangian coordinates.

We  define the geometric quantities
$$
\mathcal{A} = (D h)^{-1},
\qquad 
J = \det(Dh)
$$
and the nabla operator in Lagrangian coordinates
$$
(\nabla_{\A})_i = \A_i^j \partial_j.
$$
Direct computations show that
$$
\A_i^k \partial_k h^j = \A_k^j \partial_i h^k = \delta_i^j,
\qquad\partial_k (J \A_i^k) = 0,
$$
$$
\partial_i h^j = \delta_i^j + \partial_i \xi^j,
\qquad
\A_i^j = \delta_i^j - \A_i^k \partial_k \xi^j.
$$
In particular, we have
\beq \label{geo1}
\partial_t J = J \A_i^j \partial_j v^i,
\eeq
\beq \label{geo2}
\partial_t \A_i^j = - \A_k^j \A_i^l \partial_l v^k .
\eeq
In $\Omega_0$, 
the plasma equations in Lagrangian coordinates are then
\beq \label{L1}
\partial_t \xi = v,
\eeq
\beq \label{L2}
\rho \partial_t v + \nabla_\A \Bigl( p + {1 \over 2} |B|^2 \Bigr) = (B \cdot \nabla_\A) B,
\eeq
\beq \label{L3}
\partial_t \rho + \rho \nabla_\A \cdot v = 0,
\eeq
\beq \label{L4}
\partial_t B + B \nabla_\A \cdot v = (B \cdot \nabla_\A) v,
\eeq
\beq \label{L5}
\nabla_\A \cdot B = 0.
\eeq
On the interface $\Sigma_0$, we have the jump conditions
\beq \label{L6}
 B \cdot n = \widehat B \cdot n = 0, 
\eeq
\beq \label{L7}
\Bigl[ \Bigl[ p  + {|B|^2 \over 2} \Bigr] \Bigr] = 0.
\eeq

%%%%%%%%%%%%%%%%%%%%%%%%%%%%%%%%%%%%%%%%%%%%%%%%%%%%%%%%%%%%%%%%%%%%%%%%%%%%%

\subsection{Study of the linearized system}

%%%%%%%%%%%%%%%%%%%%%%%%%%%%%%%%%%%%%%%%%%%%%%%%%%%%%%%%%%%%%%%%%%%%%%%%%%%%%

%%%%%%%%%%%%%%%%%%%%%%%%%%

\subsubsection{Linearization }

%%%%%%%%%%%%%%%%%%%%%%%%%%

We consider a stationary solution  $(\rho_0,v_0,p_0,B_0)$ of the MHD system with $v_0 = 0$. 
Following \cite{BGT1,BGT2} and \cite{Goed} (paragraph $6.1.12$), the linearized MHD equations
in $\Omega_0$ are
\beq \label{LMHD1}
\rho_0 \partial_t v_1 = - \nabla_x p_1 + j_1 \times B_0 + j_0 \times B_1,
\eeq
\beq \label{LMHD2}
\partial_t p_1 = - v_1 \cdot \nabla_x p_0 - \gamma p_0 \nabla_x \cdot v_1,
\eeq
\beq \label{LMHD3}
\partial_t B_1 = - \nabla_x \times (v_1 \times B_0).
\eeq
% \beq \label{LMHD4}
% \partial_t \rho_1 = - \nabla_x \cdot (\rho_0 v_1).
% \eeq
Let $\xi(t,x)$ be the displacement, then at leading order we have
$$ 
v_1 = \partial_t \xi_1.
$$
The system (\ref{LMHD2},\ref{LMHD3}) leads to
\beq \label{LMHD2b}
\partial_t \Bigl[ p_1 + \xi_1 \cdot \nabla_x p_0 + \gamma p_0 \nabla_x \cdot \xi_1 \Bigr] = 0,
\eeq
\beq \label{LMHD3b}
\partial_t \Bigl[ B_1 + \nabla_x \times (\xi_1 \times B_0) \Bigr] = 0.
\eeq
% \beq \label{LMHD4b}
% \partial_t \Bigl[ \rho_1 + \nabla_x \cdot (\rho_0 \xi_1) \Bigr] = 0. 
% \eeq
These last two equations allow us to directly express $p_1$ and $B_1$ in terms of $\xi_1$ through 
\beq \label{LMHD2c}
p_1 = - \xi_1 \cdot \nabla_x p_0 - \gamma p_0 \nabla_x \cdot \xi_1 + \delta p_1,
\eeq
\beq \label{LMHD3c}
B_1 = - \nabla_x \times (\xi_1 \times B_0) + \delta B_1,
\eeq
%\beq \label{LMHD4c}
%\rho_1 = - \nabla_x \cdot (\rho_0 \xi_1) + \delta \rho_1,
%\eeq
where $\delta p_1$ and $\delta B_1$ are such that these identities hold
at $t = 0$ (and then they hold for all times). Namely,
$$
\delta p_1 = p_1(0) + \xi_1(0) \cdot \nabla_x p_0 + \gamma p_0 \nabla_x \cdot \xi_1(0),
$$
$$
\delta B_1 = B_1(0) + \nabla_x \times (\xi_1(0) \times B_0).
$$
%$$
%\delta \rho_1 = \rho_1(0) + \nabla_x \Bigl( \rho_0 \xi_1(0) \Bigr).
%$$
Finally, (\ref{LMHD1}) may be rewritten in the form
\beq \label{LMHDsimple}
\rho_0 \partial_t^2 \xi_1 = F(\xi_1)  + F_0
\eeq
where
\beln \label{defiF}
F(\xi_1) &=  \nabla_x \Bigl[ \xi_1 \cdot \nabla_x p_0 + \gamma p_0 \nabla_x \cdot \xi_1 \Bigr] 
+ B_0 \times \Bigl(\nabla_x \times \nabla_x \times (\xi_1 \times B_0) \Bigr) 
\\
&\quad - (\nabla_x \times B_0) \times \nabla_x \times  (\xi_1 \times B_0)
\eeln
% \beq \label{defiF}
% =  \nabla_x \Bigl[ \gamma p_0 \nabla_x \cdot \xi_1 \Bigr]
% + \nabla_x \Bigl[ \xi_1 \cdot \nabla_x p_0 \Bigr]
% - (\nabla_x \times B_0) \times \nabla_x (\xi_1 \times B_0)
% \eeq
and
$$
F_0 = - \nabla_x \delta p_1 - B_0 \times (\nabla_x \times \delta B_1)
+ ( \nabla_x \times B_0) \times \delta B_1.
$$
We note that $F(\xi_1)$ is a second order differential operator in $\xi_1$. 
Moreover $F_0$ is a linear operator in the initial conditions, which is of first order in $p_1$ and $B_1$, 
and of second order in $\xi_1$.
Note that $\rho^{-1} F$ is self-adjoint, as proved in \cite{Goed} (section $6.2.3$).

%%%%%%%%%%%%%%%%%%%%%%%%%%%%%

\subsubsection{The Hain-L\"ust equation}

%%%%%%%%%%%%%%%%%%%%%%%%%%%%%

We now detail (\ref{LMHDsimple}) by computing it explicitly in a suitable frame, following \cite{Goed} and in particular its chapter $9$. 
We define a local frame by
$$
e_r, \qquad
e_\perp = {(0,B_z, - B_\theta) \over | B|},
\qquad
e_\parallel = {(0,B_\theta,B_z) \over |B|}.
$$
We turn to cylindrical coordinates, take the Fourier
transform in $z$ and $\theta$, and the Fourier
transform in time, which leads to looking for solutions
of the form
$$
\xi_1(r,\theta,z,t) = \Bigl(
\xi_r(r),\xi_\theta(r),\xi_z(r) \Bigr) e^{i (m \theta + k z - \omega t)}.
$$
Note that $\lambda = - i \omega$. In these coordinates,
$$
\nabla = e_r \partial_r 
+ i g e_\perp + i f e_\parallel,
$$
with
$$
g = {1 \over |B|} \Bigl(  {m B_z \over r} - k B_\theta \Bigr) = {G \over |B|}, \qquad
f = {1 \over |B|} \Bigl(  {m B_\theta \over r} + k B_z \Bigr) = {F \over |B|}.
$$
We now define $(\xi_0,\eta_0,\zeta_0)$ as the projections of $\xi_1$ in these new coordinates, which gives
\bel
\xi_0 &= e_r \cdot \xi_1 = \xi_r, \\
\eta_0 &= i e_\perp \cdot \xi_1 = i {B_z \xi_\theta - B_\theta \xi_z \over |B|}, \\
\zeta_0 &= i e_\parallel \cdot \xi_1 = i {B_\theta \xi_\theta + B_z \xi_z \over |B|}.
\eel
%Next, we have
% $$
% Q = i f |B| \xi_0 e_r - \Bigl[(B_\theta \xi_0)' - k |B| \eta_0 \Bigr] e_\theta 
% - {1 \over r} \Bigl[ (r B_z \xi_0)' + m |B| \eta_0 \Bigr] e_z,
% $$
%$$
%p_1 = - p' \xi_1 - \gamma p \nabla \cdot \xi_1,
%\qquad 
%\nabla \cdot \xi_1 = {1 \over r}
%(r \xi_1)' + g \eta + f \zeta,
%$$
%where we removed the multiplicative factor $\exp[ i (m\theta + k z - \omega t)]$ on the right hand side
%of the two previous identities in order to simplify the reading.
This equation (\ref{LMHDsimple}) transforms into
\beq \label{mat1}
M \left( \begin{array}{c} 
\xi_0 \cr \eta_0 \cr \zeta_0 \end{array} \right) 
= - \rho \omega^2 \left( \begin{array}{c} \xi_0 \cr \eta_0 \cr \zeta_0 \end{array} \right) + \widetilde F_0
\eeq
where
$$
M = \left( \begin{array}{ccc}
\partial_r {\gamma p + |B|^2 \over r} \partial_r r
- f^2 |B|^2 - r \Bigl( {B_\theta^2 \over r^2} \Bigr)'
& 
\partial_r g (\gamma p + |B|^2) - {2 k B_\theta |B| \over r}
& \partial_r f \gamma p \cr
- {g (\gamma p + |B|^2) \over r} \partial_r r
 - {2 k B_\theta |B| \over r} &
 - g^2 (\gamma p + |B|^2) - f^2 |B|^2
 & - f g \gamma p \cr
 -{f \gamma p \over r} \partial_r r
 & - f g \gamma p & - f^2 \gamma p \cr
 \end{array} \right) 
$$  
and where $\widetilde F_0$ is the expression of $F_0$ in the new coordinates.
We refer to \cite{Goed} (chapter $9$) for the detailed computations. In the matrix $M$, to avoid multiple parentheses,
the operator $\partial_r$ acts on the whole expression on its right, up to the minus sign. 
For instance $\partial_r r$ acting on a function $\phi$ gives $\partial_r ( r \phi)$.

We now reduce (\ref{mat1}) to a single scalar equation on 
$$
\chi = r \xi_0
$$
only.
To this end, we note that we can express $\eta_0$ and $\zeta_0$ in terms of $\xi_0$ by inverting
the lower $2 \times 2$ submatrix $M_1$ of $M$ which contains no derivatives in $r$, 
% The inverse of $M_1$
% is
% $$
% M_1^{-1} =
% \begin{pmatrix}
% -\dfrac{1}{B^2\,(f^2+g^2)} & \dfrac{g}{f B^2\,(f^2+g^2)} \$$10pt]
% \dfrac{g}{f B^2\,(f^2+g^2)} & -\dfrac{1}{f^2 \gamma p} - \dfrac{g^2}{f^2 B^2\,(f^2+g^2)}
% \end{pmatrix},
% $$
which after some algebra leads to
\begin{align*}
\eta_0 &= {G [ (\gamma p + |B|^2) \rho \omega^2 - \gamma p F^2]
r \chi' + 2 k B_\theta (|B|^2 \rho \omega^2 - \gamma p F^2) \chi
\over r^2 |B|D},
\\
\zeta_0 &= { \gamma p F [ (\rho \omega^2 - F^2) r \chi'
+ 2 k B_\theta G \chi] \over r^2 |B|D},
\end{align*}
where
$$
D = \rho^2 \omega^4 - \Bigl( {m^2 \over r^2} + k^2 \Bigr)
\Bigl( (\gamma p + |B|^2) \rho \omega^2 -  \gamma p F^2 \Bigr) .
$$
It remains to write the equation on $\chi$.
Let us introduce
$$
N = (\rho \omega^2 - F^2) \Bigl( (\gamma p + |B|^2) \rho \omega^2 - \gamma p F^2 \Bigr).
$$
Inserting the previous expressions of $\eta_0$ and $\zeta_0$ into (\ref{mat1}) and after a long computation 
we obtain the classical Hain-L\"ust equation
\begin{align*} 
%{\mathcal L} \chi := 
\partial_r \Bigl[ {N \over r D} \partial_r \chi \Bigr]
& + \Bigl[ {1 \over r} (\rho \omega^2 - F^2) - \partial_r \Bigl( {B_\theta^2 \over r^2} \Bigr)
 - {4 k^2 \over r^3} {B_\theta^2 \over D} (|B|^2 \rho \omega^2 - \gamma p F^2)
\\
& + \partial_r \Bigl( {2 k \over r^2} {B_\theta G \over D} 
\Bigl( (\gamma p + |B|^2) \rho \omega^2 - \gamma p F^2 \Bigr) \Bigr) \Bigr]
\chi = 0,
\end{align*}
which is a singular second order differential equation (singular at $r = 0$).

%%%%%%%%%%%%%%%%%%%%%

\subsubsection{Boundary conditions \label{boundaryconditions}}

%%%%%%%%%%%%%%%%%%%%%

We follow \cite{Goed} (section $9.2.2$).
First, at $r = 0$, the boundary condition is 
\beq \label{condb1}
\chi(0) = 0.
\eeq
If the plasma fills the full domain (no vacuum), the boundary condition is simply 
\beq \label{condb2}
\chi(r_{ext}) = 0.
\eeq
If there is a vacuum, namely if $r_0 < r_{ext}$, 
the boundary condition at $r_{ext}$ is much more complicated.  
Let us introduce
$$
\Pi = -{N \over r D} \chi'
+ \Bigl[ {2 B_\theta^2 \over r^2} - {2 k B_\theta G \over r^2 D}
\Bigl( (\gamma p + B^2) \rho \omega^2 - \gamma p F^2 \Bigr) \Bigr] \chi
$$
and
$$
\widehat F = {m \widehat B_{\theta} \over r_0} + k \widehat B_{z}.
$$
We also introduce the modified Bessel functions $I_m$ and $K_m$.
Then, according to \cite{Goed}, the boundary condition at $r_0$ is (for $k \ne 0$)
\beq \label{condb3}
{\Pi(r_0) \over \chi(r_0)} =
{B_\theta^2(r_0) - \widehat B_\theta^2(r_0) \over r_0^2}
- {\widehat F^2(r_0) \over k r_0}
{I_m(k r_0) K_m'(k r_{ext}) - K_m(k r_0) I_m'(k r_{ext}) 
\over I_m'(k r_0) K_m'(k r_{ext}) - K_m'(k r_0) I_m'(k r_{ext})}.
\eeq
This boundary condition is a linear relation between $\chi(r_0)$ and $\chi'(r_0)$
of the form
$$
C(k,m) \chi(r_0) + C'(k,m)  \chi'(r_0) = 0
$$
for some explicit but complicated coefficients $C(k,m)$ and $C'(k,m)$ which may be chosen such that
$$
C(k,m)^2 + C'(k,m)^2 = 1.
$$

%%%%%%%%%%%%%%%%%%%%%%%%%

\subsection{Study of the $\theta-$pinch}

%%%%%%%%%%%%%%%%%%%%%%%%%

Let us consider a $\theta-$pinch configuration which satisfies the assumptions detailed in section \ref{stationary}.
In this section, we construct the Green function $G$ of the Hain-Lüst equation (\ref{mat1}) 
for such $\theta-$pinch configurations and then prove that linear instability implies nonlinear instabilities.
The main difficulty is to investigate the singularity of $G$ as $r_1$ goes to $0$ and to $r_0$.
We first focus on the case where $\lambda$ is bounded and then study the case of large $\lambda$.

%%%%%%%%%%%%%%%%

\subsubsection{Green function of the Hain-L\"ust equation}

%%%%%%%%%%%%%%%%

\begin{lemma} \label{theta} Let $\lambda = i \omega$ be fixed.
In the $\theta-$pinch configuration, if 
$$
\gamma p F^2(0) \ne \Bigl( \gamma p(0) + B^2(0) \Bigr) \rho(0) \omega^2
$$ 
and if 
$$
\rho(0) \omega^2 - F^2(0) \ne 0,$$
then the Green function $G(r_1,r)$ of the Hain-Lüst equation satisfies, for $r_1$ and $r$ different 
from $0$ and $r_0$ and for any $n \ge 0$
\beq \label{BB}
| \partial_r^n G(r_1,r) | \le C K^n n! \,  \Bigl[ 1_{r_1 < r_0/2} r^{\max(|m|-n,0)} r_1^{-|m|} 
+ 1_{r_1 > r_0/2} (r_0 - r_1)^{2 - n} \Bigr]
\eeq
if $r < r_1$ and
\beq \label{BB0}
| \partial_r^n G(r_1,r) | \le C K^n n! \,  \Bigl[1_{r_1 < r_0/2} r_1^{|m|} r^{-|m| - n} 
+ 1_{r_1 > r_0/2} (r_0 - r_1)^{2 - n} \Bigr]
\eeq
if $r > r_1$ where $C$ and $K$ are positive constants.
\end{lemma}

\begin{proof}
In the $\theta-$pinch configuration, $B_\theta = 0$,
thus  the Hain-L\"ust equation simplifies into
\beq \label{Lust1}
\partial_r \Bigl[ {N \over r D} \partial_r \chi \Bigr] +{1 \over r} (\rho \omega^2 - F^2) 
\chi = 0.
\eeq
Let $G(r_1,r)$ be its Green function, defined by 
\beq \label{Grr}
\partial_r \Bigl[ {N \over r D} \partial_r G \Bigr] +{1 \over r} (\rho \omega^2 - F^2) 
G = \delta_{r_1}.
\eeq
First, away from $0$ and $r_{0}$ and away from the point source $r = r_1$, this equation is holomorphic,
thus the Green function is holomorphic and locally bounded. 
We just have to study the singularity at $r = 0$ and $r_{0}$
Let us first investigate the singularity at $r = 0$ and assume that $r < 2 r_{0} / 3$.

\begin{itemize}

\item If $m \ne 0$, when $r$ goes to $0$, provided $\gamma p F^2 - (\gamma p + |B|^2) \rho \omega^2 \ne 0$,
$$
D \sim {m^2 \over r^2} \Bigl[ \gamma p F^2 - (\gamma p + |B|^2) \rho \omega^2 \Bigr],
$$
thus
\beq \label{ND}
{N \over r D} \sim - r {\rho \omega^2 - F^2 \over m^2}.
\eeq
We now apply Fuchs' theory to this equation.
The indicial equation is $s^2 = m^2$, hence the exponents are $s = \pm |m|$. Note that the difference between these two exponents is an integer. Frobenius theory then provides the existence of two independent solutions of the Hain-L\"ust equation, of the form
$$
\chi_1(r) = r^{+|m|} \widetilde \chi_1(r) 
$$
and
$$
\chi_2(r) = r^{-|m|} \widetilde \chi_2(r)
+  \alpha r^{+|m|} \widetilde \chi_1(r) \log r,
$$
where $\widetilde \chi_1$ and $\widetilde \chi_2$ are holomorphic
with $\widetilde \chi_1(0) = 1$, $\widetilde \chi_2(0) = 1$
and $\alpha \in \mathbb{C}$.
We note that $\chi_1(0) = 0$ whereas $\chi_2(r)$ is not bounded near $0$.
We thus look for $G(r_1,r)$ in the form
\beq \label{G1}
G(r_1,r) = a_1(r_1) \chi_1(r), \quad \hbox{if} \quad r < r_1,
\eeq
\beq \label{G2}
G(r_1,r) = a_2(r_1) \chi_1(r) + a_3(r_1) \chi_2(r),
\quad \hbox{if} \quad r > r_1
\eeq
since $G$ must be bounded at $r = 0$.
Moreover, $G$ must be continuous and $N \partial_r G / r D$ must have 
a unit jump at $r = r_1$.
This gives the system
$$
a_1(r_1) \chi_1(r_1) - a_2(r_1) \chi_1(r_1) - a_3(r_1) \chi_2(r_1) = 0,
$$
$$
a_1(r_1) \chi_1'(r_1) - a_2(r_1) \chi_1'(r_1) - a_3(r_1) \chi_2'(r_1) = - {r_1 D \over N}.
$$
The boundary condition at $r_0$ gives a relation between $a_2$ and $a_3$, of the form
$a_2(r_1) = C_\star a_3(r_1)$, where the constant $C_\star$ is discussed in section \ref{boundaryconditions}.
This gives
$$
a_1 = C_0 {C_\star \chi_1(r_1) +  \chi_2(r_1) \over \Delta} {D \over N} r_1,
$$
$$
a_2 = C_\star^2 {\chi_1(r_1) \over \Delta} {D \over N} r_1,
$$
where
$$
\Delta = C_\star \Bigl( \chi_2'(r_1) \chi_1(r_1) - \chi_2(r_1) \chi_1'(r_1) \Bigr).
$$
We note that 
$$
\Delta  \sim - {2 |m| C_\star \over r_1} 
$$
as $r_1$ goes to $0$. Thus, using (\ref{ND}), we have
$$
a_1 = O(r_1^{-|m|}), \qquad a_2 = O(r_1^{|m|}).
$$
As a consequence, the Green function $G(r_1,r)$ is bounded even as $r_1$ goes to $0$.
%The study of $\partial_r G$ is similar. Note that $\partial_r^2 G$ is in general not bounded, 
%but is only bounded by $C r_1^{-1}$.

Let us turn to the proof of the bounds on the derivatives. First for $r < r_1$,
since $\widetilde \chi_1$ is analytic, we get
$$
|\widetilde \chi_1^{(k)}(r)| \le k!\, M C^k
$$
for some constants $M$ and $C$. Using Leibniz’s rule, we can show that
$$
\chi_1^{(n)}(r) = \sum_{j=0}^{n} \binom{n}{j} (r^{|m|})^{(j)} \widetilde \chi_1^{(n-j)}(r).
$$
Because \(r^{|m|}\) is a monomial, its derivatives are zero for \(j > |m|\). For \(j \le |m|\),
$$
\bigl| (r^{|m|})^{(j)} \bigr| \le \frac{|m|!}{(|m|-j)!}\, r^{|m|-j}.
$$
Hence, if $n \le |m|$,
\begin{align*}
|\chi_1^{(n)}(r)| &\le \sum_{j=0}^{\min(|m|,n)} \binom{n}{j} \frac{|m|!}{ (|m|-j)!}\, r^{|m|-j}\; (n-j)!\, M C^{n-j}.
\\
&\le n! M r^{|m| - n} \, \sum_{j=0}^{\min(|m|,n)} \frac{|m|!}{j ! (|m|-j)!} C^{n-j} r^{n-j}
\\
&\le n! M r^{|m|-n} C^n
\end{align*}
up to a change of $C$.
If $n > |m|$, the same computations lead to $|\chi_1^{(n)}(r) | \lesssim n! M C^n$.

Using the bound for \(a_1\) and the bound for \(\chi_1^{(n)}\), we get for any \(r \in (0, r_1)\) that
\beq \label{BB1}
|\partial_r^n G(r_1,r)|
\le |a_1(r_1)|\; |\chi_1^{(n)}(r)|
\le C_1 K^n n! \, r^{\max(|m|-n,0)} r_1^{-|m|}
\eeq
for some constant $C_1$.

The case $r > r_1$ can be treated using similar arguments, and we obatin
\beq \label{BB2}
|\partial_r^n G(r_1,r) | \le C_1 K^n n! \, r_1^{|m|} r^{-|m| - n} .
\eeq

\item If $m = 0$, then $D$ remains smooth at $r = 0$.
In this case, the Frobenius equation on the exponents is $s (s - 2)= 0$.
Thus, there exist two independent solutions of the Hain-L\"ust equation of the form
$$
\chi_1(r) = r^2 \widetilde \chi_1(r) 
$$
and
$$
\chi_2(r) = \widetilde \chi_2(r) + \alpha r^2 \widetilde \chi_1(r) \log r,
$$
where $\widetilde \chi_1$ and $\widetilde \chi_2$ are holomorphic
with $\widetilde \chi_1(0) = 1$, $\widetilde \chi_2(0) = 1$
and $\alpha \in \mathbb{C}$. The computation of the Green function can be done as in the previous paragraph.

\end{itemize}

We now turn to the study of the singularity at \(r_0\), where the density \(\rho\) is not smooth, and focus on the region \(r>r_0/2\). More precisely, near \(r_0\), the density \(\rho\) is analytic on a domain of the form
$$
\Delta=\Bigl\{\frac{r_0}{2}<r<r_0,\quad |\Im r|<\delta\Bigr\}
$$
for some \(\delta>0\). The solutions \(\chi_1,\chi_2\) are defined on \(\Delta\). Let us study their behaviour as \(r\to r_0\).
As the coefficients of the Hain-L\"ust equation are bounded, $\chi$, $\chi'$ and $\chi''$ are bounded when $r \to r_0$.
Then, using the holomorphy of $\chi''$, each derivative leads to a loss of a factor $r_0 - r$.
This proves the required estimate near \(r_0\) and completes the proof of the lemma.
\end{proof}

\begin{remark}
The relation $\gamma p F^2 = (\gamma p + B^2) \rho \omega^2$ is the dispersion relation of the Alfvén waves.
In this case, the Hain-Lüst degenerates into a zeroth order operator.
\end{remark}

We now consider the regime $|\omega|\to+\infty$. We observe that
$$
\frac{N}{rD} = \frac{r(\rho\omega^2-F^2)} {r^2\rho\omega^2
[(\gamma p+|B|^2)-\gamma pF^2\rho^{-1}\omega^{-2}]^{-1} -m^2-k^2r^2 },
$$
changes behavior at the radius
$$
r_c = \frac{1}{|\omega|} \sqrt{ \frac{m^2}{\rho}(\gamma p+|B|^2) }.
$$
For $r\ll r_c$,
$$
\frac{N}{rD} \sim -\frac{r}{m^2}(\rho\omega^2-F^2),
$$
while for $r\gg r_c$,
$$
\frac{N}{rD} \sim \frac{\gamma p+|B|^2}{r}.
$$
At $r=0$, the Frobenius exponents are $s=\pm |m|$.  Consequently there exist two independent solutions
$$
\chi_1(r,\omega)=r^{|m|}\widetilde\chi_1(r,\omega),
$$
$$
\chi_2(r,\omega)=r^{-|m|}\widetilde\chi_2(r,\omega)
+\alpha r^{|m|}\widetilde\chi_1(r,\omega)\log r.
$$
Introducing $y=r/r_c$, the scaled equation shows that $\widetilde\chi_1$ and $\widetilde\chi_2$ are uniformly bounded in $\omega$ for $0\le r / r_c \lesssim 1$.

For $r\ge r_c$, we use the WKBJ method. Writing
$$
\chi(r,\omega)=e^{\omega\phi(r,\omega)},
$$
and substituting into the leading-order equation
$$
\partial_r\Bigl(
\frac{\gamma p+|B|^2}{r}\partial_r\chi
\Bigr)
+
\frac{\rho\omega^2-F^2}{r}\chi
=
0,
$$
we obtain the Riccati equation
$$
A(\phi')^2+\rho + \frac1\omega \Bigl[ A\phi'' + \Bigl(A'-\frac{A}{r}\Bigr)\phi' \Bigr] - \frac{F^2}{\omega^2}
= 0,
$$
where
$$
A=\gamma p+|B|^2.
$$
At leading order,
$$
A(\phi')^2+\rho=0,
$$
hence
$$
\phi_\pm' \sim \pm i\sqrt{\frac{\rho}{A}}.
$$
The two WKB solutions are therefore
\beq \label{lala}
\chi_\pm(r,\omega) = e^{\omega\phi_\pm(r,\omega)}
= \exp\Bigl( \pm i\omega\int^r\sqrt{\frac{\rho(t)}{A(t)}}\,dt \Bigr)
\Bigl(1+O(|\omega|^{-1})\Bigr).
\eeq
We choose the branch of the square root $\sqrt{\rho/A}$
so that the solution $\chi_+$ decays exponentially as $r$ increases, while $\chi_-$ decays exponentially as $r$ decreases.

We now construct global solutions $\tilde \chi_\pm(r,\omega)$ on $0 \le r \le r_0$.
For $r \ll r_c$, we define $\tilde \chi_+$ by $\tilde \chi_+(r,\omega) = \chi_1(r,\omega)$.
We then extend it for $r \gg r_c$ by a linear combination of $\chi_\pm$
$$
\tilde \chi_+(r,\omega) = \alpha_{++} \chi_+(r,\omega) + \alpha_{+-} \chi_-(r,\omega).
$$
The study of $\partial_r \tilde \chi_+(r,\omega) / \tilde \chi_+(r,\omega)$ near $r_c$ shows that $\alpha_{++} \ne 0$.

To construct $\tilde \chi_-$, we start with the linear combination of $\chi_\pm$ which satisfies the boundary condition
at $r_0$ and extended it to the left of $r_c$. As $\omega$ is not an eigenvalue, 
$\tilde \chi_-$ is not a multiple of $\chi_1$ for $r < r_c$ and thus $\tilde \chi_-$ and $\tilde \chi_+$ are two independent solutions.

We then explicitly construct the Green function by
$$
G(r_1,r) = {r_1 D \over N} {\tilde \chi_-(r_1) \tilde \chi_+(r) \over W}
$$
if $r < r_1$ and by
$$
G(r_1,r) = {r_1 D \over N} {\tilde \chi_+(r_1) \tilde \chi_-(r) \over W} 
$$
if $r > r_1$, where $W$ is the Wronskian of $\tilde \chi_\pm$.
These formulas are the generalizations of (\ref{G1}) and (\ref{G2}) for large $\lambda$.
In view of (\ref{lala}), $W$ is of order $\omega$.
In particular,
\beq \label{bounddd}
| G | \lesssim {C \over |\lambda|} .
\eeq
The estimates on the derivatives are similar.

%%%%%%%%%%%%%%%%

\subsubsection{Study of the Green function }

%%%%%%%%%%%%%%%%

Let us now study bounds on the resolvent (\ref{LMHDsimple}) or equivalently (\ref{mat1}).
We introduce the singular generator
$$
\mathcal G_{\mathrm{sing}}[\xi](\rho) =
\sum_{k\ge 0} \Bigl\| \bigl(r(r_0-r)\partial_r\bigr)^k \xi \Bigr\|_{L^\infty} \frac{\rho^k}{k!}.
$$

\begin{lemma} \label{theoremsingular2}
Under the assumptions of Lemma~\ref{theta}, there exists $\rho_1>0$, independent of $f$, such that
$$
\|r^{-1} \xi \|_{L^\infty} + \mathcal G_{\mathrm{sing}}[\partial_r \xi](\rho_1)
\lesssim {1 \over |\lambda|} \mathcal G_{\mathrm{sing}}[F_0](\rho_1).
$$
\end{lemma}

\begin{proof} 
We first study the Green function of $\xi_0$, $\eta_0$ and $\zeta_0$.
We recall that $\chi = r \xi_0$ and that $\eta_0$ and $\zeta_0$ are obtained from $\xi_0$ 
through multiplication by smooth function or application of the operator $r \partial_r$.
As a consequence, the Green function $G_2$ of $\xi_0$, $\eta_0$ and $\zeta_0$ satisfies
\beq \label{BB}
| \partial_r^n G_2(r_1,r) | \le C K^n n! \,  \Bigl[ r^{\max(|m|-n-1,-1)} r_1^{-|m|} + (r_0 - r_1)^{2 - n} \Bigr]
\eeq
if $r < r_1$ and
\beq \label{BB0}
| \partial_r^n G_2(r_1,r) | \le C K^n n! \,  \Bigl[ r^{-|m|-n-1} r_1^{|m|} + (r_0 - r_1)^{\min(2 - n,0)} \Bigr].
\eeq
The most singular term of the first coefficient $M_{11}$ of $M$ is $r^{-1} (\gamma p + |B|^2) \partial_r^2$,
thus when we construct the Green function $G_3$ of (\ref{mat1}), or equivalently of (\ref{LMHDsimple}),
we gain a factor $r_1$, leading to
\beq \label{BBG3}
| \partial_r^n G_3(r_1,r) | \le C K^n n! \,  \Bigl[ r^{\max(|m|-n-1,-1)} r_1^{1-|m|} + (r_0 - r_1)^{2 - n} \Bigr]
\eeq
if $r < r_1$ and
\beq \label{BB0G3}
| \partial_r^n G_3(r_1,r) | \le C K^n n! \,  \Bigl[ r^{-|m|-n-1} r_1^{|m|+1} + (r_0 - r_1)^{\min(2 - n,0)} \Bigr].
\eeq
Now
$$
\xi(r) = \int_0^{r_0} G_3(r_1,r) F_0(r_1)\,dr_1,
$$
which leads to for bounded $\lambda$ to
$$
|\xi(r)| \le C r\| F_0 \|_{L^\infty}, 
\qquad
|\partial_r \xi(r)| \le C \| F_0 \|_{L^\infty}.
$$
For large $\lambda$, we use (\ref{bounddd}).
The Lemma is then a consequence of Lemma \ref{greensingular}.
\end{proof}

%%%%%%%%%%%%%%%%%%%%%%%%%%%%%%%%%%%%%%%%%%%%%%%%%%%%%%%%%%%%%%%%%

\subsubsection{Nonlinear instability for the $\theta$-pinch}
\label{sec:theta-nonlinear}

%%%%%%%%%%%%%%%%%%%%%%%%%%%%%%%%%%%%%%%%%%%%%%%%%%%%%%%%%%%%%%%%%

We now apply the general instability theorem to the $\theta$-pinch configuration.

\begin{theorem}[Nonlinear instability of the $\theta$-pinch]
\label{nonl}
Let $(B_0,u_0=0,p_0)$ be a stationary $\theta$-pinch solution satisfying
Assumption~\ref{assumptionAtheta}. Assume that the linearized Hain--L\"ust
operator has finite spectral radius $\sigma>0$. Then the plasma column is
nonlinearly unstable in the sense of Theorem~\ref{linearnonlinear}.
\end{theorem}

\begin{proof}
We verify the three assumptions of Theorem~\ref{linearnonlinear}.

Let us check the first Assumption.
The linearized operator on $(0,r_0)$ with the boundary conditions
\eqref{condb1}--\eqref{condb3} has a compact resolvent. 
%This follows from the explicit Green function estimates proved in Lemma~\ref{theta} and from the
%standard one-dimensional Sturm--Liouville compactness argument. 
Hence its spectrum is discrete and consists of eigenvalues of finite multiplicity.
Thus there exists an eigenvalue $\lambda_0$ such that 
$\Re\lambda_0 >\sigma / 2$.
The corresponding eigenfunction
$\chi_{\lambda_0}$ is bounded at $r=0$, so it is proportional
to the regular Frobenius solution
$r^{|m|}\widetilde\chi_1(r)$
where $\widetilde\chi_1$ is holomorphic.
The other components of the full eigenvector are
obtained from $\chi_{\lambda_0}$ by algebraic operations and by
$r^{-1}\partial_r(r^2\,\cdot)$ and are thus also holomorphic.

Assumption $2$ is a direct consequence of the Lemma \ref{theoremsingular2}.
It remains to check the Assumption $3$.
In Lagrangian coordinates, the nonlinear terms of the MHD involve only one derivative, thus
Assumption~3 is satisfied.
All the assumptions of Theorem~\ref{linearnonlinear} are therefore verified. 
Applying that theorem, the stationary $\theta$-pinch is nonlinearly unstable.
\end{proof}

%%%%%%%%%%%%%%%%%%%%%%%%%%%%%%%%%%%%%%%

\subsection{Study of the $z$-pinch} \label{sec:zpinch}

%%%%%%%%%%%%%%%%%%%%%%%%%%%%%%%%%%%%%%%

We now turn to the case of the $z$-pinch. Throughout this section,
we consider a $z$-pinch configuration satisfying Assumption~\ref{assumptionAZ}.

For $B_z=0$, the Hain--L\"ust equation becomes
\begin{equation}
\label{Hain2}
\partial_r \Bigl[ \frac{N}{rD}\partial_r\chi\Bigr] + H \chi = 0,
\end{equation}
where
\begin{align*}
H &=
\frac{\rho\omega^2-F^2}{r}
-\partial_r\Bigl(\frac{B_\theta^2}{r^2}\Bigr)
-\frac{4k^2}{r^3}\frac{B_\theta^2}{D}
\bigl(|B|^2\rho\omega^2-\gamma pF^2\bigr)
\\
&\quad + \partial_r\Bigl[
\frac{2k}{r^2}\frac{B_\theta G}{D}
\bigl((\gamma p+|B|^2)\rho\omega^2-\gamma pF^2\bigr)
\Bigr].
\end{align*}
We now analyze the singularity of this equation at $r=0$.
As $B_\theta(0)=0$ and $B_\theta$ is analytic, we may write
$B_\theta(r)=r \widetilde B_\theta(r)$,
with $\widetilde B_\theta$ analytic and $\widetilde B_\theta(0)=B_\theta'(0)$. Since
$G=-kB_\theta$ and $F = r^{-1} mB_\theta$,
the functions
$r^{-2} B_\theta^2$, $r^{-2} B_\theta G$, and $F$
are all analytic near $r=0$. Moreover, for $m\neq0$,
$F(0)=mB_\theta'(0)$.
Finally, the quantities
$|B|^2=B_\theta^2$ and $|B|^2\rho\omega^2-\gamma pF^2$
are at least $O(r^2)$ near $r=0$, so after division by $D$, the corresponding terms in $H$ are bounded. Thus there exists a smooth function $\widetilde H$ such that
$$
H = \frac{\rho\omega^2-F^2}{r} + \widetilde H.
$$
This is the only singular term in the zeroth-order part of the Hain--L\"ust equation.

Let us first consider the case $m\neq0$. Near $r=0$,
$$
D \sim - \frac{m^2}{r^2} \gamma p\bigl(\rho\omega^2-F^2\bigr).
$$
Indeed, $B_\theta=O(r)$, so the terms involving $|B|^2$ are of lower order compared with $\gamma p$. 
Under the nondegeneracy condition
$$
\rho(0)\omega^2-F(0)^2\neq0,
$$
we also have
$$
N
=
(\rho\omega^2-F^2) \bigl((\gamma p+|B|^2)\rho\omega^2-\gamma pF^2\bigr) \sim \gamma p(\rho\omega^2-F^2)^2.
$$
Therefore
$$
\frac{N}{rD} \sim - \frac{r}{m^2}\bigl(\rho\omega^2-F^2\bigr),
$$
and
$$
H \sim \frac{\rho\omega^2-F^2}{r}.
$$
Thus the Hain--L\"ust equation has exactly the same Frobenius structure near $r=0$ as in the $\theta$-pinch case. The indicial exponents are again $s=\pm |m|$,
and there are two independent solutions
$$
\chi_1(r)=r^{|m|}\widetilde\chi_1(r),
$$
$$
\chi_2(r)=r^{-|m|}\widetilde\chi_2(r) +\alpha r^{|m|}\widetilde\chi_1(r)\log r.
$$
In particular, the Green function satisfies the same bounds as in Lemma~\ref{theta}.

The case $m=0$ is analogous. Then $F=0$, and
$$
\frac{N}{rD} \sim \frac{C}{r},
$$
so the Frobenius exponents are $0$ and $2$. The Green function estimates are obtained in the same way.

We now state the following instability result for the $z$-pinch.

\begin{theorem}[Nonlinear instability for the $z$-pinch]
\label{non2}
Let $(B_0,u_0=0,p_0)$ be a stationary $z$-pinch solution satisfying Assumption~\ref{assumptionAZ}. Assume that the plasma column is linearly unstable, in the sense that the spectral radius $\sigma$ defined in \eqref{definitionsigma} satisfies 
$0 < \sigma < + \infty$.
Then the plasma column is nonlinearly unstable in the sense of Theorem~\ref{linearnonlinear}.
\end{theorem}

\begin{proof}
The proof is essentially identical to that of Theorem~\ref{nonl} for the $\theta$-pinch, because the singular structure of the Hain--L\"ust equation near $r=0$ and near $r=r_0$ is the same.
Let us check the various assumptions.

The spectral radius $\sigma$ is finite, and by hypothesis there exists an eigenvalue $\lambda_0$ with $\Re\lambda_0>\sigma/2$. The corresponding eigenfunction is regular at $r=0$, hence proportional to $\chi_1$ which is holomorphic.

For bounded $\lambda$, the Green function estimates give the desired weighted resolvent bound. For large $\lambda$, the same WKB construction as for the $\theta$-pinch applies because the leading coefficient
$N/rD$
has the same asymptotic limits
$$
\frac{N}{rD}\sim-\frac{r}{m^2}(\rho\omega^2-F^2)
\quad\text{for }r\ll r_c,
$$
and
$$
\frac{N}{rD}\sim\frac{\gamma p+|B|^2}{r}
\quad\text{for }r\gg r_c.
$$
Thus the resolvent estimate
$$
\mathcal G_{\mathrm{sing}}\bigl[(\lambda\rho_0-L)^{-1}f\bigr](\rho_1)
\le
\frac{C}{|\lambda|}
\mathcal G_{\mathrm{sing}}[f](\rho_1)
$$
holds uniformly for large $\lambda$.

The nonlinearity is the same as in the $\theta$-pinch case.
Thus all the assumptions of the weighted version of Theorem~\ref{linearnonlinear} are satisfied, and the nonlinear instability follows.
\end{proof}

\begin{remark}
We refer to \cite{BGT2} for a detailed study of $\sigma$ in the $z$-pinch case. In particular, Theorem~2.1 of \cite{BGT2} 
states that $\sigma$ is always positive, so the plasma is always linearly unstable.
\end{remark}

%%%%%%%%%%%%%%%%%%%%%%%%%%%%%%%%%%%%%%%%%%%%%%%%%%%%%%%%%%%

\subsection{An ill-posedness result\label{secparticular}}

%%%%%%%%%%%%%%%%%%%%%%%%%%%%%%%%%%%%%%%%%%%%%%%%%%%%%%%%%%%

The assumptions \ref{assumptionAtheta} and \ref{assumptionAZ} imply that $p$ and $\rho$ do not vanish
for $0 \le r \le r_0$.
In this section we discuss the case where $p$ exactly vanishes at $r = r_0$, as for instance 
in the previously introduced case
\beq \label{particularbis}
p(r) = c^2 \Bigl(1 - {r^2 \over r_{0}^2} \Bigr), \quad B_\theta(r) = c {r \over r_{0}},
\quad
u(r) = 0.
\eeq
We will prove that the inviscid MHD system is ill-posed near such stationary states.
More precisely, let us consider a solution $(\rho,u,B)$ of the inviscid MHD system
such that $\rho(t,\cdot) > 0$ inside a domain $\Omega(t)$, and $\rho(t,\cdot) = 0$ outside.
We assume that $\rho$ has no jump on $\partial \Omega(t)$, namely that $\rho$ is continuous on both
sides of $\partial \Omega(t)$ with $\rho = 0$ on $\partial \Omega(t)$.

Using (\ref{M2}), this implies that $A = 0$ on $\partial \Omega(t)$ where
$$
A(t,x) = \nabla \Bigl( p + {1 \over 2} |B|^2 \Bigr) - (B \cdot \nabla) B.
$$
A direct computation shows that this holds true for the particular solution (\ref{particularbis}).

\begin{theorem}
Let $(p_0,u_0 = 0,B_0)$ be a $z-$pinch  stationary solution of the inviscid MHD system
satisfying (\ref{cyl}),
such that $p_0(r) > 0$ for $0 \le r < r_0$ and $p_0(r) = 0$ for $r \ge r_0$, $p'(r_0) \ne 0$ and such that $B_0(r_0) \ne 0$.
Assume that $p_0$ and $B_0$ are $C^\infty$ on $r \le r_0$.
Then, locally near this stationary solution, the inviscid MHD system is ill-posed in the following sense:
let $m$ be large enough, then
for any $s$ arbitrarily large and any $\eps > 0$ arbitrarily small, there exists an initial data
$(p(0),B(0),u(0))$ such that the corresponding vector $A$ vanishes on $\partial \Omega(0)$, such that
$$
\| p(0) - p_0 \|_{H^s} + \| B(0) - B_0 \|_{H^s}  + \| u(0) \|_{H^s} \le \eps
$$
but so that there exists no local in time solution $(p(t,x),B(t,x),u(t,x)) \in L^{\infty}(0,T,H^m)$ 
on the support of $p$, even for arbitrarily small $T > 0$.
\end{theorem}

\begin{proof}
Let us consider a solution $(p(t,x),B(t,x),u(t,x))$ such that $p > 0$ on $\Omega(t)$ and $p = 0$ on
its complementary. We know that $\Omega(t)$ moves with the velocity $u$.
As (\ref{M1}) may be rewritten under the form
$$
\partial_t \rho + (u \cdot \nabla) \rho = - \rho \nabla \cdot u
$$
and as all the functions are smooth inside $\Omega(t)$, we see that $\rho = 0$ on $\partial \Omega(t)$.
Thus, using (\ref{M2}), $A(t,x)$ also vanishes on $\partial \Omega(t)$. 

Let us consider a characteristic $X(t)$ which starts from $\partial \Omega(0)$,
namely
$$
\dot X(t) = u(t,X(t))
$$
with $X(0) \in \partial \Omega(0)$. We note that $\dot X(0) = 0$ provided $u(0,X(0)) = 0$ (which will
be checked later).
As $A(t,x)$ vanishes at the boundary,
$$
A(t,X(t)) = 0
$$
for any time $t$.
Differentiating in time and using $\dot X(0) = 0$, as $A$ is smooth in $\Omega(t)$, we get
\beq \label{ptA}
\partial_t A(0,X(0)) = 0.
\eeq
We now construct an initial data $(p(0),B(0),u(0))$ such that at $t = 0$, (\ref{ptA}) does not hold true.
Let $(r,\theta,z)$ be the cylindrical coordinates with orthonormal basis $(e_r,e_\theta,e_z)$.
We look for initial data such that, 
\begin{itemize}
\item At $t=0$, the pressure depends only on $r$ and satisfies
\begin{equation}\label{prad}
p(0,r)=p(r),\qquad
p(r_0)=0,
\qquad
p(r)>0\quad (r<r_0).
\end{equation}
\item At $t=0$, the magnetic field is azimuthal and depends only on $r$
\begin{equation}\label{Brad}
B(0,r)=b(r)e_\theta.
\end{equation}
\item At $t=0$, the velocity is purely radial and depends only on $r$, and it vanishes at the vacuum boundary:
\begin{equation}\label{urad}
u(0,r)=a(r)e_r,
\qquad
a(r_0)=0.
\end{equation}
\end{itemize}
Under these assumptions, the divergence condition $\nabla\cdot B=0$ is satisfied,
\[
\nabla\Bigl(p+\frac12|B|^2\Bigr)
=
\Bigl(p'+bb'\Bigr)e_r,
\]
where primes denote derivatives with respect to $r$, and
\[
(B\cdot\nabla)B
=
b\,\frac{1}{r}\partial_\theta\bigl(b\,e_\theta\bigr)
=
\frac{b^2}{r}\partial_\theta e_\theta
=
-\frac{b^2}{r}e_r,
\]
therefore
\begin{equation}\label{Ar}
A(0,r) = A_r(0,r) = \Bigl(p'+bb'+\frac{b^2}{r}\Bigr)e_r .
\end{equation}
Thus $A$ has only a radial component. Note that for the particular solution $(p_0,B_0,u_0 = 0)$, 
$A(0,r_0) = 0$ is a direct consequence of (\ref{cyl})

From the mass equation and the equation of state, we have
\begin{equation}\label{dtp}
\partial_t p = -u\cdot\nabla p-\gamma p\,\nabla\cdot u
= -a p'-\gamma p\Bigl(a'+\frac{a}{r}\Bigr).
\end{equation}
%At $r=r_0$, using $a(r_0)=0$ and $p(r_0)=0$, we obtain $\partial_t p(r_0)=0$.
Differentiating \eqref{dtp} with respect to $r$ gives
\[
(\partial_t p)'
=
-a'p'-ap''
-\gamma p'\Bigl(a'+\frac{a}{r}\Bigr)
-\gamma p\Bigl(a''+\frac{a'}{r}-\frac{a}{r^2}\Bigr)
= -(1+\gamma)a'p',
\]
since $a=0$ and $p=0$ at $r = r_0$.
Next we use the induction equation
\[
\partial_t B = \nabla\times(u\times B)
% \]
% For $u=a(r)e_r$ and $B=b(r)e_\theta$,
% \[
% u\times B
% =
% a b\,e_z.
% \]
% In cylindrical coordinates,
% \[
% \nabla\times\bigl(ab\,e_z\bigr)
% =
% -\bigl(ab\bigr)'e_\theta
= -\bigl(a'b+ab'\bigr)e_\theta.
\]
Thus
\[
\partial_t b
=
-a b'-b a'.
\]
% At $r=r_0$, since $a=0$,
% \begin{equation}\label{dtb0}
% \partial_t b(r_0)
% =
% -b a'.
% \end{equation}
Differentiating $\partial_t b$ with respect to $r$ yields
\[
(\partial_t b)'
=
-a b''-2a'b'-b a''.
\]
At $r=r_0$,
\begin{equation}\label{dtbprime}
(\partial_t b)'(r_0)
=
-2a'b'-b a''.
\end{equation}
Thus, the differentiation of \eqref{Ar} gives
\[
\partial_t A_r = (\partial_t p)' + b'\partial_t b + b(\partial_t b)' + \frac{2b\,\partial_t b}{r}.
\]
At the point $X=(r_0,z)$, we get
\begin{align}
\partial_t A_r(0,X)
%&=
%-(1+\gamma)a'p' + b'\bigl(-b a'\bigr) + b\bigl(-2a'b'-b a''\bigr) + \frac{2b}{r_0}\bigl(-b a'\bigr)
%\nonumber\\
&=
-(1+\gamma)a'p' - 3a'bb' - b^2 a'' - \frac{2b^2 a'}{r_0}.
\end{align}
We choose $p = p_0$, $B = B_0$, so that $A(0,r_0) = 0$ using (\ref{Ar}).
We then choose $a(r)$ smooth, close to $0$ in $H^s$, with $a(r_0) = 0$, $a'(r_0) = 0$ and $a''(r_0) \ne 0$.
Then $\partial_t A_r(0,r_0) \ne 0$, which is in contradiction with the existence of a smooth solution 
in small time.
\end{proof}

%%%%%%%%%%%%%%%%%%%%

\subsection*{Acknowledgments}  

%%%%%%%%%%%%%%%%%%%%

D. Bian was partially supported by the NSF of China (No. 12271032)
and would like to thank Y. Guo and I. Tice for helpful discussions. This work was partially done
while D. Bian participated in the program “Mathematical developments in Geophysical
fluid dynamics” and she is grateful for the hospitality of IHP.

%%%%%%

%%%%%%%%%%%%%%%%%%%%%%%%%%%

\end{document}